%% file: WalledBrauerAnnulator.tex
\documentclass[preprint,12pt]{elsarticle}

\usepackage{graphicx}
\usepackage{ifthen}
 \usepackage{epsfig}
\usepackage{amssymb}
\usepackage{amsmath}
\usepackage{dsfont}
\usepackage{pb-diagram}
\usepackage{bbm}
\usepackage{fp}

\usepackage{pgfopts}
\usepackage{mathdots}
\usepackage[all]{xy}

\usepackage[vcentermath,enableskew]{youngtab}

\usepackage{caption}
\usepackage{subcaption}

\newcommand{\sgn}[1]{\mathrm{sgn}(#1)}
\newcommand{\Std}{\operatorname{Std}}
\newcommand{\tab}{\operatorname{tab}}
\newcommand{\Tab}{\mathcal{T}}
\newcommand{\Semistd}{\mathcal{T}_0}
\renewcommand{\tt}{\mathfrak t}
\renewcommand{\ss}{\mathfrak s}
\newcommand{\uu}{\mathfrak u}
\newcommand{\vv}{\mathfrak v}
\newcommand{\oo}{\mathfrak o}

\renewcommand{\aa}{\mathfrak a}
\newcommand{\bb}{\mathfrak b}
\newcommand{\cc}{\mathfrak c}
\newcommand{\ZZ}{\mathbb Z}
\newcommand{\Endo}{E}

\newcommand{\shape}{\operatorname{shape}}
\newcommand{\fil}{\operatorname{fill}}
\newcommand{\op}{\operatorname{anti}}
\newcommand{\Res}{\operatorname{Res}}
\newcommand{\Rat}{\operatorname{Rat}}
\newcommand{\I}{\bf I}

\newlength{\tangle}
\newlength{\tangletwo}
\newlength{\tanglethree}
\newlength{\productone}
\newlength{\producttwo}
\newlength{\productthree}

\newlength{\unit}
\newcommand{\htangle}[2]{\settoheight{\tangle}
  {\epsfysize=#2 \epsfbox{figures/#1}}
  \addtolength{\tangle}{-1ex}
  \raisebox{-.5\tangle}{\epsfysize=#2 \epsfbox{figures/#1}}}
\newcommand{\wtangle}[2]{\settoheight{\tangle}
  {\epsfxsize= #2\epsfbox{figures/#1}}
  \addtolength{\tangle}{-1ex}
  \raisebox{-.5\tangle}{\epsfxsize=#2 \epsfbox{figures/#1}}}

\newcommand{\tuple}[3]{ \left(
    \ifthenelse{\equal{#1}{0}}{\emptyset}{
      \settoheight{\tangle}{\epsfbox{figures/#1}}
      \raisebox{-#3\tangle}{
        \raisebox{#3\unit}{\epsfysize=#3\tangle \epsfbox{figures/#1}}}}\;,
    \ifthenelse{\equal{#2}{0}}{\emptyset}{
      \settoheight{\tangletwo}{\epsfbox{figures/#2}}
      \raisebox{-#3\tangletwo}{
        \raisebox{#3\unit}{\epsfysize=#3\tangletwo
          \epsfbox{figures/#2}}}}
    \;
  \right)}

\newcommand{\triple}[5]{ \left(
    \FPmul{\productone}{#4}{#5}
    \FPmul{\producttwo}{\productone}{.5}
    \FPmul{\productthree}{#4}{.5}
    \ifthenelse{\equal{#1}{0}}{\emptyset}{
      \settoheight{\tangle}{\epsfbox{figures/#1}}
      \raisebox{\producttwo\unit}{
        \raisebox{-#4\tangle}{
          \raisebox{\productthree\unit}{
            \epsfysize=#4\tangle \epsfbox{figures/#1}}}}}\;,
    \ifthenelse{\equal{#2}{0}}{\emptyset}{
      \settoheight{\tangletwo}{\epsfbox{figures/#2}}
      \raisebox{\producttwo\unit}{\raisebox{-#4\tangletwo}{
          \raisebox{\productthree\unit}{\epsfysize=#4\tangletwo
            \epsfbox{figures/#2}}}}}\;, 
    \ifthenelse{\equal{#3}{0}}{\emptyset}{
      \settoheight{\tanglethree}{\epsfbox{figures/#3}}
      \raisebox{\producttwo\unit}{\raisebox{-#4\tanglethree}{
          \raisebox{\productthree\unit}{\epsfysize=#4\tanglethree
            \epsfbox{figures/#3}}}}} \quad
  \right)}

 \usepackage{amsthm}
\newtheorem{thm}{Theorem}
\newtheorem{lemma}[thm]{Lemma}

\newtheorem{prop}[thm]{Proposition}
\newtheorem{cor}[thm]{Corollary}
\newtheorem{defn}[thm]{Definition}
\newtheorem{exmp}[thm]{Example}

\newtheorem{remark}[thm]{Remark}

\journal{}

\begin{document}

\begin{frontmatter}



\title{A cell filtration of mixed tensor space}


\author[s]{F. Stoll}
\ead{stoll@mathematik.uni-stuttgart.de}

\author[s]{M. Werth}
\ead{werth@mathematik.uni-stuttgart.de}

\address[s]{Institut f\"ur Algebra und Zahlentheorie,
  Universit\"at Stuttgart,
  Pfaffenwaldring 57,
  70569 Stuttgart,
  Germany}

\begin{abstract}
We construct a cellular basis of the walled Brauer algebra which has similar properties as the Murphy basis of the group algebra of the symmetric group. In particular, the restriction of a cell module to a certain subalgebra can be easily described via this basis. Furthermore, the mixed tensor space possesses a filtration by cell modules -- although not by cell modules of the walled Brauer algebra itself, but by cell modules of its image in the algebra of endomorphisms of mixed tensor space.

\end{abstract}


\end{frontmatter}



\input epsf

\input{introduction}
\input{preliminaries}
\input{dimend}

\input{basis}

\input{generators}

\input{restriction}
\input{annihilatorbasis}
\input{filtration}

\input bib

\end{document}

%% file: introduction.tex
\section{Introduction}

Throughout, let $n$ be a natural number, $r$, $s$ and $m$  non-negative integers and let $R$ be a commutative ring. For $x\in R$,  
the walled Brauer algebra $B_{r,s}(x)$ is the subalgebra of the Brauer algebra $B_{r+s}(x)$ spanned by certain diagrams, the walled Brauer diagrams. Let $V$ be a free $R$-module of rank $n$ and let $V^*=\mathrm{Hom}_R(V,R)$. Then the walled Brauer algebra for the parameter $x=n$ acts on the mixed tensor space $V^{\otimes r}\otimes {V^*}^{\otimes s}$ and satisfies Schur-Weyl duality together with the action of the universal enveloping algebra $U$ of the general linear algebra (see \cite{bchlls,koike,turaev,dipperdotystoll1,dipperdotystoll2}).  

This situation is very similar to (in fact a generalization of) the classical Schur-Weyl duality where both the group algebra $R\mathfrak{S}_m$ of the symmetric group and $U$ act on the ordinary tensor space $V^{\otimes m}$
(\cite{schur,weyl,green}).  The group algebra of the symmetric group is a cellular algebra in the sense of Graham and Lehrer (\cite{grahamlehrer}) with a cellular basis (the Murphy basis, see \cite{murphy}) which has remarkable properties and one may ask, if the walled Brauer algebra has a cellular basis with similar properties: 

\begin{itemize}
\item If a cell module of $R\mathfrak{S}_m$ is restricted to the subalgebra $R\mathfrak{S}_{m-1}$, then this module has a filtration by cell modules of $R\mathfrak{S}_{m-1}$.  One can obtain this filtration canonically from the Murphy basis and the involved combinatorics are quite simple. Can one construct a basis of the walled Brauer algebra with the same property?
\item The annihilator of $R\mathfrak{S}_m$ on the ordinary tensor space is a cell ideal (\cite{haerterich}). Is the same true for the annihilator of the walled Brauer algebra on the mixed tensor space? Or is there at least a cellular basis of the walled Brauer algebra such that for each $n$, a subset of this basis spans the annihilator? Or is the factor algebra at least again a cellular algebra?
\item The ordinary tensor space has a filtration by cell modules of $R\mathfrak{S}_m$, in fact it has a filtration with  $U$-$R\mathfrak{S}_m$-bimodules which are tensor products of a certain  $U$-module with a cell module for $R\mathfrak{S}_m$. Is the mixed tensor space filtered by corresponding cell modules/bimodules as well?
\end{itemize}
 
The known cellular bases do not have these properties and we will define a new cellular basis of the walled Brauer algebra, such that we can  give  positive answers to these three questions. 
Before we describe the answers, a subtle point has to be noted:
there cannot exist  a generic cellular structure on the walled Brauer algebra such that the mixed tensor space is filtered by cell modules with respect to this cellular structure. By `generic' we mean a cellular structure for $R=\mathbb{Z}[x]$ which can be specialized to cellular structures for arbitrary $R$ and $x$.
 
Consider for example the case $r=s=2$. The  walled Brauer algebra over $\mathbb{C}(x)$ with parameter $x$ is semisimple, there are four irreducible modules of  dimension $1$, one of dimension $2$ and one of dimension $4$. Although, there might be different (generic) cellular structures on the walled Brauer algebra, the dimensions of the cell modules must coincide with the dimensions of the irreducible modules in the semimsimple case. If $n=2$ and $R=\mathbb{C}$, then the action of the walled Brauer algebra on the mixed tensor space is semisimple since the action of $U$ is semisimple. In this case, the mixed tensor space decomposes into a direct sum of three $3$-dimensional, one $2$-dimensional and five  one-dimensional irreducible modules. Thus, it follows that the mixed tensor space does not have a cell filtration, since the only cell module that might have a $3$-dimensional constituent is $4$-dimensional, but  this cell module is not  semisimple. In the same way,  the annihilator is not a cell ideal, otherwise the walled Brauer algebra modulo the annihilator would be again a cellular algebra and semisimple with irreducible modules of dimension $1,2$ or $4$, a contradiction. 
A similar phenomenon has been observed in \cite{henkepaget} for the action of the Brauer algebra on tensor space. 
   
The example above shows that if a cellular basis of the walled Brauer algebra includes a basis of the annihilator (when $x$ is specialized to $n$) then there might be elements $\lambda$ in the poset $(\Lambda,\leq)$ attached to the cell datum	 such that there are basis elements $C_{S,T}^\lambda$  which lie in the annihilator as well as basis elements which do not lie in the annihilator. Note that only for $x=n$ the walled Brauer algebra acts on some mixed tensor space and then the annihilator is defined. We will construct a (generic) basis of the walled Brauer algebra such that under specialization a basis element is in the annihilator iff $S$ or $T$ is not in a certain subset of $M(\lambda)$ (depending on $n$). Figure \ref{fig:illustrationofannihilator} illustrates this in the case $r=s=2$. The big squares stand for the elements of the poset $\Lambda$. Each big square (which corresponds to $\lambda\in \Lambda$) consists of an  array of small  squares, the number of rows/columns is equal to the cardinality  of $M(\lambda)$. So, each small square represents one basis element. The basis elements in the annihilator for various $n$ are those in the shaded areas. 

\begin{figure}
  \begin{center}
    \begin{subfigure}[b]{0.2\linewidth}
      \begin{center}
        \epsfxsize=.4\textwidth
        \epsfbox{figures/blockx.21}
        \caption*{$n = \dim V \geq 4$}
      \end{center}
    \end{subfigure}
    \begin{subfigure}[b]{0.2\linewidth} 
      \begin{center}
        \epsfxsize=.4\textwidth
        \epsfbox{figures/blockx.22}
        \caption*{$n = 3$}
      \end{center}
    \end{subfigure}
    \begin{subfigure}[b]{0.2\linewidth}
      \begin{center}
        \epsfxsize=.4\textwidth
        \epsfbox{figures/blockx.23}
        \caption*{$n = 2$}
      \end{center}
    \end{subfigure}
    \begin{subfigure}[b]{0.2\linewidth}
      \begin{center}
        \epsfxsize=.4\textwidth
        \epsfbox{figures/blockx.24}
        \caption*{$n = 1$}
      \end{center}
    \end{subfigure}
  \end{center}
  \caption{}
  \label{fig:illustrationofannihilator}
\end{figure}

The example suggests that the walled Brauer algebra modulo its annihilator is again a cellular algebra by taking cosets of the non-shaded basis elements, with cell datum obtained from those of the walled Brauer algebra by eventually omitting elements in the sets $M(\lambda)$.  In fact, the constructed basis has this property and thus the factor algebra of the walled Brauer algebra modulo the annihilator is a cellular algebra. We show that we can filter the mixed tensor space by cell modules of the factor algebra. 

Instead of this basis containing a basis of the annihilator, we use basis  elements which can be easily defined and still have the property that a subset of the cosets of basis elements forms a cellular basis of the factor. Restriction of cell modules can be canonically described with this basis.
Summarizing the main results, 
\begin{itemize}
\item  Theorem~\ref{thm:basis} establishes the existence  of the cellular basis of the walled Brauer algebra, 
\item Theorem~\ref{thm:restriction} shows that restriction of cell modules is canonical with respect to this basis.
\item In Theorem~\ref{thm:filtration}, we give a filtration of the mixed tensor space by certain bimodules, and thus cell modules of the algebra of $U$-endomorphisms. 
\end{itemize}

This paper is organized as follows: in Section~\ref{section:preliminaries}, we supply details about the action of the walled Brauer algebra on the mixed tensor space and cellular algebras. In Section~\ref{section:combinatorics} we give a combinatorial description of the rank of the annihilator of the walled Brauer algebra on the mixed tensor space starting with the trivial $U$- (or $\mathbb{C}GL_n$-)module and successively tensoring with $V$ and $V^*$. This leads to appropriate cell data for the cellular structures of both the walled Brauer algebra and its image in the algebra of $U$-endomorphisms (for $x=n$). Similar to the symmetric group case,  
 the elements of indexing sets in such a cell datum are paths of tuples of partitions. In 
Section~\ref{section:basis} we replace these paths by triples of certain tableaux and define the new cellular basis of the walled Brauer algebra. Section~\ref{section:restriction} investigates the behaviour of cell modules under  restriction to the subalgebra  $B_{r,s-1}(x)$ for $s\geq 1$ or $B_{r-1,0}$ for $r\geq 1, s=0$. It turns out that the isomorphism between a factor and a cell module of the subalgebra can be defined by mapping a basis to a basis.  The corresponding map on the triples of tableaux labeling  the basis elements can be described quite easily. In Section~\ref{section:annihilatorbasis} we define another (weakly) cellular basis of the walled Brauer algebra such that for each $n$ and $x$ specialized to $n$, a subset of this basis is a basis of the annihilator. This shows that the algebra of $U$-endomorphisms inherits the desired cellular structure. Section~\ref{section:ordinary_tensor_space} recalls some facts about the ordinary tensor space which are needed to finally show in Section~\ref{section:filtration_mixed_tensor_space} that the mixed tensor space is filtered by cell modules and moreover by bimodules.


%% file: preliminaries.tex
\section{Preliminaries}\label{section:preliminaries}
Throughout, let $r$, $s$ and  $m$ be fixed nonnegative integers. 
Furthermore, let  $n$ be  a natural number. Let 
$R$ be a commutative ring with $1$ and $x\in R$. 

\subsection{The walled Brauer algebra}
A \emph{Brauer diagram} $d$ is a graph with $2m$ vertices such that each
vertex is connected to precisely	 one other vertex. 
Usually, the vertices are located in two rows, $m$ vertices in an
upper row and $m$ vertices in a bottom row, and the edges are drawn
inside of this rectangle (see
Figure~\ref{figure:brauer-diagram}). Edges connecting vertices of one
row are 
called \emph{horizontal}, edges connecting vertices of different rows
are called \emph{vertical}. 

\begin{figure}[h!]
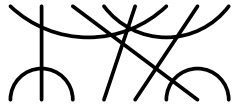

	\centering
  \wtangle{examples.1}{3cm}
	\caption{A \emph{Brauer diagram}}\label{figure:brauer-diagram}
\end{figure}

\begin{defn}[\cite{brauer}]
  The \emph{Brauer algebra} $B_m(x)$ is the $R$-algebra which is free as an
  $R$-module with basis $\{d\mid d\text{ a Brauer diagram with
  }2m\text{ vertices}\}$. The multiplication is given by concatenation
  of diagrams, closed cycles are deleted by multiplying $x$ (see
Figure~\ref{figure:multiplication}). 
\end{defn}

\begin{figure}[h!]
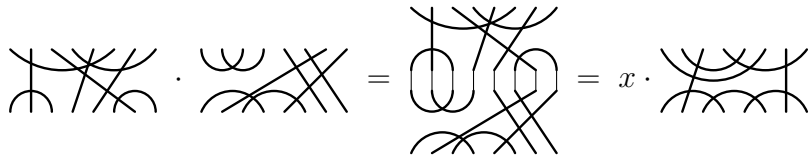

  \centering
  $\wtangle{examples.1}{2cm} \;\cdot\; \wtangle{examples.2}{2cm}
  \;=\; \wtangle{examples.3}{2cm} \;=\; x\cdot \wtangle{examples.4}{2cm}$
  \caption{Multiplication in the Brauer
    algebra}\label{figure:multiplication} 
\end{figure}

Note that each permutation $\pi$ of $m$ letters corresponds to a
Brauer diagram without horizontal edges, namely the diagram that
connects the $i$-th vertex in the upper row with the $\pi(i)$-th
vertex in the bottom row.  Thus, the group algebra $R\mathfrak{S}_m$
is a subalgebra of the Brauer algebra.
 
Now let $m=r+s$. Consider a vertical wall between the $r$-th and
$r+1$-st vertex in each row. A Brauer diagram is called \emph{walled
  Brauer diagram} if all horizontal edges cross the wall and all
vertical edges do not cross the wall (see
Figure~\ref{figure:walleddiagram}).

\begin{figure}[h!]
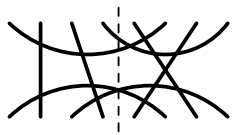

  \centering
  \wtangle{examples.11}{3cm}
  \caption{A walled Brauer diagram}\label{figure:walleddiagram}
\end{figure}

\begin{defn}[\cite{bchlls,koike,turaev}]
  The \emph{walled Brauer algebra} $B_{r,s}(x)$ is the subalgebra of
  $B_m(x)$  generated by the walled Brauer diagrams. In
  fact, the walled Brauer diagrams form an $R$-basis of $B_{r,s}(x)$.
\end{defn}
 Again, $R(\mathfrak{S}_r\times \mathfrak{S}_s)$ is a subalgebra of the walled
Brauer algebra. 
Another way to decide whether or not a given Brauer diagram is a
walled diagram (with $r$ and $s$ given) is the following: Draw 
arrows pointing downwards at the first $r$ vertices in both rows. At
the other vertices draw  arrows pointing upwards. Then a Brauer
diagram is a walled diagram if and only if
the edges  can be oriented consistently with
the arrows at the vertices (see
Figure~\ref{figure:orienteddiagram}). 

\begin{figure}[h!]
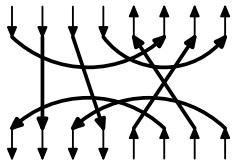

	\centering
  \wtangle{examples.12}{3cm}
	\caption{A walled Brauer diagram with
          orientation}\label{figure:orienteddiagram} 
\end{figure}

Now, the notion of a walled Brauer diagram can be generalized to diagrams with not necessarily the same numbers of vertices in the top and bottom row. If two sequences 
$\bf I$ and $\bf J$ with entries in $\{\downarrow, \uparrow\}$ are given, then in a generalized Brauer diagram of type $(\bf I,\bf J)$ the vertices in the top row are labeled by the entries of $\bf I$, the vertices in the bottom row are labeled by $\bf J$ and the edges of the diagram again connect the vertices consistently with
the arrows. Therefore, a walled Brauer diagram is the same as  a generalized diagram of type $((\downarrow^r,\uparrow^s),(\downarrow^r,\uparrow^s))$. If two diagrams of type $(\bf I,\bf J)$ and $(\bf J,\bf K)$ are given, then these diagrams can be concatenated to a diagram of type $(\bf I,\bf K)$.

\subsection{The mixed tensor space}
Let $V$ be an $R$-free module of rank $n$ with basis
$\{v_1,\ldots,v_n\}$. Let $I(n,m)$ be the set of $m$-tuples with
entries in $\{1,\ldots,n\}$. Then a basis of $V^{\otimes m}$ is given by
$\{v_{\mathbf i}=v_{i_1}\otimes v_{i_2}\otimes\cdots \otimes v_{i_m}\mid
\mathbf{i}=(i_1,\ldots,i_m)\in I(n,m) \}$. 
The Brauer algebra
$B_m(n)$ acts on $V^{\otimes m}$ which is called the \emph{tensor
  space}. If $d$ is a Brauer diagram, then the matrix (acting from the
right) of the
endomorphism induced by $d$ with respect to the basis above is
obtained in the following way: for $\mathbf i,\mathbf j\in I(n,m)$
write $i_1,\ldots,i_m$ at the vertices in the top row of the diagram
and $j_1,\ldots,j_m$ at the vertices in the bottom row. The matrix
entry at the position $\mathbf i,\mathbf j$ is $1$ if for all edges in
$d$  both ending vertices have the same number and $0$ otherwise. 
This action extends the action of the symmetric group
$\mathfrak{S}_m$ permuting the components of the tensor product.

Let $V^*=\mathrm{Hom}_R(V,R)$, which is as an $R$-module isomorphic to
$V$. Let $\{v_i^*\mid i=1,\ldots,n\}$ be the basis of $V^*$ dual to
$\{v_i\mid i=1,\ldots,n\}$. By identifying $v_i^*$ and $v_i$, 
the walled Brauer algebra $B_{r,s}(n)$   acts as a subalgebra
of the Brauer algebra on
$V^{\otimes r}\otimes {V^*}^{\otimes s}$ called the \emph{mixed tensor
  space}. 
More generally, a generalized walled Brauer diagram of type $(\bf I,\bf J)$ induces a homomorphism from $V_{\bf I}$ to $V_{\bf J}$ where $V_{\bf I}$ is a tensor product of $V$'s and $V^*$'s such that $\downarrow$ stands for $V$ and $\uparrow$ stands for $V^*$.  Concatenation of those diagrams and deleting cycles by multiplication with $n$ is compatible with composition of homomorphisms.

Let $L$ be a linear combination of such generalized Brauer diagrams
such that the number and orientation of vertices in the top row of each diagram
coincide, the same for the bottom row. By a diagram involving a box
containing $L$ we mean the linear combination we get by taking the
corresponding linear combination of diagrams obtained by replacing the
box by the smaller diagram. We will use the term diagram also for
diagrams containing boxes. 

\begin{exmp}
  Let $L= 3\cdot \htangle{generalbrauer.1}{.5cm}-2\cdot
  \htangle{generalbrauer.2}{.5cm}$. Then

  \[
  \htangle{generalbrauer.3}{1.5cm}=3\cdot 
  \htangle{generalbrauer.4}{1.5cm}-2\cdot 
  \htangle{generalbrauer.5}{1.5cm}
  \] 
\end{exmp}

Note that placing two diagrams next to each other corresponds to
taking the tensor product of the corresponding maps. 
This has the following consequence:
\begin{remark}
  If a diagram contains a box inducing the zero map, then this diagram
  itself induces the zero map. 
\end{remark}  
  
  If $R=\mathbb{C}$ the quantum group algebra
$\mathbb{C}GL_n(\mathbb{C})$ also acts on the
mixed tensor space and this action commutes with the action of the
walled Brauer algebra. We have the  following famous result:

\begin{thm}[Schur-Weyl duality, \cite{bchlls}]\label{theorem:Schur-Weyl-duality}
  Let $B_{r,s}(n)$ be the walled
  Brauer algebra over $\mathbb{C}$.  
  \begin{enumerate}
  \item The $\mathbb{C}$-algebra homomorphism
    \[
    B_{r,s}(n)\to \mathrm{End}_{\mathbb{C}GL_n(\mathbb{C})}(V^{\otimes r}\otimes
    {V^*}^{\otimes s})
    \]
    is surjective.
\item The $\mathbb{C}$-algebra homomorphism 
  \[
    \mathbb{C}GL_n(\mathbb{C})\to \mathrm{End}_{B_{r,s}(n)}(V^{\otimes r}\otimes
    {V^*}^{\otimes s})
    \]
    is surjective.

\end{enumerate}  
\end{thm}
 Let $U_{\ZZ}$ be the $\ZZ$-form of the universal enveloping algebra of the Lie algebra $\mathfrak{gl}_n$ defined as in \cite{carterlusztig} and let $U=U_R= R\otimes_{\ZZ} U_{\ZZ} $. Then is follows from the results in \cite{dipperdotystoll1} and \cite{dipperdotystoll2} for the special case $q=1$, that 
 this result holds more generally when $\mathbb{C}$ is replaced by 
    a commutative ring $R$ with identity and the group algebra
    $\mathbb{C}GL_n(\mathbb{C})$ is replaced by $U_R$.

\subsection{Cellular algebras}
\begin{defn}\cite{grahamlehrer}\label{defn:cellular_algebra}
  A \emph{cellular algebra} over $R$ is an $R$-algebra $A$ together
  with a   partially ordered set $(\Lambda,\leq)$,
  for each $\lambda\in\Lambda$ a
  finite set $M(\lambda)$ and an $R$-basis 
  \[B=\{C_{S,T}^\lambda\mid \lambda\in\Lambda,S,T\in M(\lambda)\}\]
  of $A$ such that the following conditions hold: 
  \begin{enumerate}
  \item[C1:] The map $*:A\to A: C_{S,T}^\lambda\mapsto
    (C_{S,T}^\lambda)^*=C_{T,S}^\lambda$  linearly extends to an
    anti-involution of $A$. 
  \item[C2:] If $\lambda\in\Lambda$ and $T,T'\in M(\lambda)$ then for each
    $a\in A$ there exist $r_a(T,T')\in R$ such that
    \[
    C_{S,T}^\lambda a\equiv \sum_{T'\in M(\lambda)} r_a(T,T')
    C_{S,T'}^\lambda\mod A(<\lambda)
    \]
    for $S\in M(\lambda)$
    where $A(<\lambda)$ is the $R$-submodule of $A$ generated by
    $\{C_{S'',T''}^\mu\mid \mu<\lambda,S'',T''\in M(\mu)\}$ 
  \end{enumerate}
  The basis $B$ is called \emph{cellular basis}. 
\end{defn}

\begin{remark}[\cite{goodmangraber}]
  Condition~C1~in Definition~\ref{defn:cellular_algebra} can be
  weakened to the following condition without loosing the results of
  \cite{grahamlehrer}: 
  \begin{enumerate}
  \item[C1':]  
    There is an anti-involution $*:A\to A$ such that
    \[{C_{S,T}^\lambda}^*\equiv C_{T,S}^\lambda\mod A(<\lambda).\]
  \end{enumerate}
  If $A$ is an $R$- algebra satisfying the same conditions as a
  cellular algebra except for C1 which is replaced by C1', then we call
  $A$ a \emph{weakly cellular algebra}.
 	If $2$ is invertible in $R$, then a weakly cellular algebra is a cellular algebra.
\end{remark}

One of the most important examples for cellular algebras is the group
algebra of the symmetric group $R\mathfrak{S}_m$.
Since we will use a cellular basis of this algebra, we 
recall the construction of such a basis
due to Murphy (\cite{murphy}).

A \emph{composition} $\lambda$ of $m$ is a sequence
$\lambda=(\lambda_1,\lambda_2,\ldots)$ of nonnegative integers whose
sum is $m$. We write $\lambda\models m$. 
Repeated occurrences of the same integer are indicated by exponents. 
If $\lambda$ and $\mu$ are compositions of $m$, we say that $\lambda$
\emph{dominates} $\mu$ and write $\lambda\unrhd\mu$ if for all $l\geq 1$ we
have $\sum_{i=1}^l\lambda_i\geq \sum_{i=1}^l\mu_i$. 

If $\lambda$ is a composition
of $m$ such that $\lambda_1\geq \lambda_2\geq\ldots$, then we say that
$\lambda$ is a \emph{partition} of $m$ and write $\lambda\vdash m$.
The unique partition of $0$ is denoted by $\emptyset$. 
 The
set of partitions of $m$ is denoted by $\Lambda(m)$.
The \emph{Young diagram} $[\lambda]$ of a partition is the set
$\{(i,j)\in\mathbb{N}^2\mid 1\leq 
j\leq\lambda_i, 1\leq i\}$. 

If $\lambda$ is a partition of $m$, then a $\lambda$-\emph{tableau}
$\tt$ is a bijection $[\lambda]\to\{1,\ldots,m\}$,
$\lambda$ is called the \emph{shape} of $\tt$ and is denoted by 
$\shape(\tt)$. Tableaux are
often depicted 
as an array of boxes, one box for each element of $[\lambda]$ at the
corresponding position (such that $(1,1)$ is in the upper left corner) and we
write the number corresponding to the position into the box. 
Figure~\ref{figure:tableau}
shows a tableau of shape $(4,3,2)$. 
\begin{figure}[h!]
  \centering
  \epsfbox{figures/youngtableaux.0}
  \caption{A $(4,3,2)$-tableau}\label{figure:tableau}
\end{figure}

A tableau is called \emph{row-standard} if the entries in each row are
increasing from left to right, it is called \emph{column-standard} if
the entries in each column are increasing downwards and  
\emph{standard} if it is both row- and column-standard. Let $\tab(\lambda)$ be the set of $\lambda$-tableaux and 
$\Std(\lambda)$ be the set of standard $\lambda$-tableaux.
Let $\tt^\lambda$ be the $\lambda$-tableau where the
numbers from $1$ to $m$ are written into the boxes row by row. Then
$\tt^\lambda$ is a standard $\lambda$-tableau (see
Figure~\ref{figure:initialtableau}).  

\begin{figure}[h!]
  \centering
  \epsfbox{figures/youngtableaux.1}
  \caption{$\tt^\lambda$ with
    $\lambda=(4,3,2)$}\label{figure:initialtableau} 
\end{figure}

The symmetric group $\mathfrak{S}_m$ acts from the right on the set of
$\lambda$-tableaux by place permutation. If
$\tt$ is a $\lambda$-tableau, let
$d(\tt)\in\mathfrak{S}_m$ be the unique element such that
$\tt.d(\tt)=\tt^\lambda $. 
Let $\mathfrak{S}_\lambda$ be the row-stabilizer of
$\tt^\lambda$, so $\mathfrak{S}_\lambda\cong
\mathfrak{S}_{\lambda_1}\times  \mathfrak{S}_{\lambda_2}\times 
 \cdots$. If $w\in \mathfrak{S}_m$,  let
 $\sgn{w}$ be the sign of the permutation $w$ and $w^*=w^{-1}$. Then $*$
 can be extended to an anti-automorphism of $R\mathfrak{S}_m$.

For a partition $\lambda$ let
$y_\lambda=\sum_{w\in\mathfrak{S}_\lambda}\sgn{w}w$. Furthermore, if
$\ss$ and $\tt$ are $\lambda$-tableaux, let
$m^\lambda_{\ss,\tt}=d(\ss)^*y_\lambda
d(\tt)$. Then we have
\begin{thm}[\cite{murphy}]\label{thm:murphy}
  The group algebra $R\mathfrak{S}_m$ of the symmetric group is a
  cellular algebra with cellular basis
  \[
  \{m^\lambda_{\ss,\tt}\mid
  \lambda\in\Lambda(m),\ss,\tt\in\Std(\lambda)\} .
  \]
  In the notation of Definition~\ref{defn:cellular_algebra},
  $(\Lambda,\leq)$   is the set of partitions $(\Lambda(m),\unrhd)$
  ordered  by
  the dominance order. 
  $M(\lambda)$ is the set $\Std(\lambda)$ of standard
  $\lambda$-tableaux and the anti-automorphism $*$ is defined above.  
\end{thm} 

Note that  the elements $m^\lambda_{\ss,\tt}$ are
defined for arbitrary $\lambda$-tableaux which are not necessarily
standard. The proof of Theorem~\ref{thm:murphy} relies on the fact
that the $m^\lambda_{\ss,\tt}$ (with
$\ss,\tt$ arbitrary $\lambda$-tableaux) clearly span
$R\mathfrak{S}_m$. Thus one has to show that if $\ss$ or
$\tt$ are not standard, then
$m^\lambda_{\ss,\tt}$ can be written as a linear
combination of such elements involving tableaux which are `dominant'
with respect to some appropriate ordering.

It is easy to see and we will use this below, that
if $\tt$ and $\tt'$ are $\lambda$-tableaux
such that the set of entries of corresponding rows coincide, then
$m^\lambda_{\ss,\tt}=\pm
m^\lambda_{\ss,\tt'}$, the same holds for tableaux on the left side.
In particular, each 
$m^\lambda_{\ss,\tt}$ is a linear combination of 
$m^\lambda_{\ss',\tt'}$ with
$\ss'$ and $\tt'$ row-standard.

Let $\tt$ be a row-standard $\lambda$-tableau for a composition
$\lambda$ of $m$. If $1\leq i\leq m$, let
$\tt\downarrow i$ be the tableau obtained by restricting the
corresponding bijection $   \{1,\ldots,m\}\leftrightarrow[\lambda]$ to
$\{1,\ldots i\}$ . So $\shape(\tt\downarrow i)$ is a composition of
$i$. If $\ss$ is a $\mu$-tableau for some composition $\mu$ of $m$, we
define $\tt\unrhd\ss$ if and only if $\shape(\tt\downarrow
i)\unrhd\shape(\ss\downarrow i)$ for all $i=1,\ldots,m$. Then we have
\begin{prop}[\cite{murphy}]\label{prop:murphy}
  Let $\ss$ and $\tt$ be row-standard  $\lambda$-tableaux for a
  partition $\lambda$ 
  of $m$ such that
  $\ss$ is not standard. Then  $m_{\ss,\tt}^\lambda$ is congruent
  modulo $R\mathfrak{S}_n(\rhd\lambda)$ to a
  linear combination of $m_{\ss',\tt}^\lambda$
  with $\ss'\rhd\ss$. A similar statement holds if $\tt$ is not standard.

  Moreover, if   $\ss$ and $\tt$ are row-standard  $\lambda$-tableaux,
  then $m_{\ss,\tt}^\lambda$ is a linear combination of
  basis elements $m_{\ss',\tt'}^\mu$ with $\ss',\tt'$ standard, 
  $\ss'\unrhd\ss$ and $\tt'\unrhd\tt$. 
\end{prop}

We will construct a weakly cellular basis of $B_{r,s}(x)$ such that if
$x$ is specialized to $n$, then the annihilator
$\mathrm{ann}_{B_{r,s}(n)}(V^{\otimes r}\otimes {V^*}^{\otimes s})$ of
the walled Brauer algebra on the mixed tensor space has a basis
consisting of a subset of this weakly cellular basis. In particular,
$B_{r,s}(n)/\mathrm{ann}_{B_{r,s}(n)}(V^{\otimes r}\otimes
{V^*}^{\otimes s})$ is again a weakly cellular algebra.




%% file: dimend.tex
\section{The rank of the annihilator}\label{section:combinatorics}
In this section, we will describe a combinatorial index set for a
basis  of the annihilator of the
walled Brauer algebra
on mixed tensor space. The next two propositions which hold also in the quantized case show that
the annihilator is in fact $R$-free with $R$-free complement 
and the rank does not depend on
$R$.  
\begin{prop}[\cite{haerterich}]\label{prop:haerterich}
  The annihilator in the group algebra  $R\mathfrak{S}_m$ of the
  symmetric group on
  $V^{\otimes m}$ is $R$-free with basis  
  \[
  \{m^\lambda_{\mathfrak{s},\mathfrak{t}}\mid
  \lambda\in\Lambda(m),\lambda_1>
  n,\mathfrak{s},\mathfrak{t}\in\Std(\lambda)\}. 
  \]
\end{prop}

\begin{remark}\label{remark:action_of_the_walled_Brauer_algebra}
  In particular if $n'>n$, then $y_{(n')}$ acts as zero and hence each
  diagram with a box containing $y_{(n')}$ acts as zero. 
\end{remark}

\begin{prop}[\cite{dipperdotystoll1}]\label{prop:dipperdotystoll}
  There is an $R$-isomorphism between $R\mathfrak{S}_{r+s}$ and the
  walled Brauer algebra $B_{r,s}(n)$, that maps the annihilator in
  $R\mathfrak{S}_{r+s}$ on the tensor space bijectively to the
  annihilator in $B_{r,s}(n)$ on the mixed tensor space. In particular
  the annihilator in the walled Brauer algebra is as well $R$-free
  with an $R$-free complement. 
\end{prop}

Thus, a basis of the annihilator of the walled Brauer algebra on mixed
tensor space is indexed by pairs of standard tableaux for partitions
$\lambda$ with $\lambda_1>n$. Anyhow, this combinatorial 
index set is related to the
group algebra of the symmetric group and its action on tensor space,
but not to the walled Brauer algebra. 
Propositions~\ref{prop:haerterich} and \ref{prop:dipperdotystoll} show
that it is enough to find a combinatorial description of the rank of the
annihilator for 
$R=\mathbb{C}$. 

The rank of the annihilator can be computed once we know the dimension
of the image of the representation  $B_{r,s}(n)\to
\mathrm{End}_{\mathbb{C}}(V^{\otimes r}\otimes 
{V^*}^{\otimes s})$ which is equal to
$\mathrm{End}_{\mathbb{C}GL_n(\mathbb{C})}(V^{\otimes r}\otimes
{V^*}^{\otimes s})$ 
by
Theorem~\ref{theorem:Schur-Weyl-duality}.
To determine this dimension one can decompose the semisimple
$\mathbb{C}GL_n(\mathbb{C})$-module $V^{\otimes r}\otimes
{V^*}^{\otimes s}$ into simple submodules 
as in \cite{stembridge}.

\begin{defn}
  Let $\lambda\vdash (r-k)$ and $\mu\vdash (s-k)$ for some nonnegative
  integer $k$.
  A \emph{path} $Y$ to $(\lambda,\mu)$ is a sequence
  $Y=(Y_0=(\emptyset,\emptyset),Y_1,Y_2,\ldots,Y_{r+s}=(\lambda,\mu))$
  where $Y_i=(\lambda^{(i)},\mu^{(i)})$ 
  for $0\leq i\leq r+s$ is a pair of partitions such that
  \begin{itemize}
  \item For $i\leq r$, $\mu^{(i)}=\emptyset$ is the empty partition
    and $[\lambda^{(i)}]$ is obtained from $[\lambda^{(i-1)}]$ by adding one box.
  \item For $r+1\leq i\leq r+s$,  either $\lambda^{(i)}=\lambda^{(i-1)}$ and
    $[\mu^{(i)}]$ is obtained from $[\mu^{(i-1)}]$ by adding one box
    or $\mu^{(i)}=\mu^{(i-1)}$ and
    $[\lambda^{(i)}]$ is obtained from $[\lambda^{(i-1)}]$ by removing one box.    
  \end{itemize}
\end{defn}
If $Y$ is a path, we write
$Y=(\emptyset\to Y_1\to Y_2 \to\dots\to   Y_{r+s})$. Given $Y$, let
$\max(Y)$ be the maximum of the set
$\{\lambda^{(i)}_1+\mu^{(i)}_1\mid i=1,\ldots,r+s\}$. We obviously have
$\max(Y)\leq r+s$.  To the  tuples $(\lambda,\mu)$ of partitions
 with $\lambda_1+\mu_1\leq n$ 
 one can define  pairwise non-isomorphic  simple rational
$\mathbb{C}GL_n(\mathbb{C})$-modules $V_{\lambda,\mu}$  such that
\begin{prop}[\cite{stembridge}]\label{prop:stembridge}
  We have:
  \[
  V^{\otimes r}\otimes {V^*}^{\otimes s}\cong
  \bigoplus_{(\lambda,\mu)}
  V_{\lambda,\mu}^{\oplus n_{\lambda,\mu}}
  \]
  where $n_{\lambda,\mu}$ is  
  equal to the number of
  paths $Y$ to $(\lambda,\mu)$ with $\max(Y)\leq n$. 
\end{prop}
 
Note that our partitions are the transposed partitions of Stembridge's
partitions in \cite{stembridge}. Proposition~\ref{prop:stembridge}
can be as well obtained using the 
Littlewood-Richardson rule 
(\cite{littlewoodrichardson},\cite{fultonharris}).   
 Together with the previous results, we immediately obtain
\begin{cor}\label{cor:rank}
  \begin{enumerate}
  \item The algebra $\mathrm{End}_{U}(V^{\otimes r}\otimes
    {V^*}^{\otimes s})$ is $R$-free. Its rank is equal to the number of
    pairs $(Y,Z)$ where $Y,Z$ are paths to $(\lambda,\mu)$ for
    partitions $\lambda\vdash (r-k)$ and $\mu\vdash (s-k)$ 
    with $\max(Y),\max(Z)\leq n$. 
  \item The rank of the walled Brauer algebra  $B_{r,s}(x)$ is equal
    to the number of pairs $(Y,Z)$ where $Y,Z$ are paths to
    $(\lambda,\mu)$ with $\lambda\vdash (r-k)$ and $\mu\vdash (s-k)$.
  \item The rank of the annihilator in the walled Brauer algebra on
    mixed tensor space is equal to the number of   pairs $(Y,Z)$ of
    paths to $(\lambda,\mu)$ such that $\max(Y)>n$ or $\max(Z)>n$. 
  \end{enumerate}
\end{cor}
\begin{proof}
  If $M$ is a semisimple module over some $\mathbb{C}$-algebra $A$,
  and $M\cong \oplus_i S_i^{\otimes n_i}$ where the $S_i$ are pairwise
  non-isomorphic simple modules, then we have
  $\dim_{\mathbb{C}}\mathrm{End}_A(M)=\sum_i n_i^2$, thus
   the first part follows. Since the rank of the walled Brauer algebra
   does not depend on $R$ and the parameter $x$ and since the walled
   Brauer algebra acts faithfully for $n\geq r+s$ we obtain the second
   part. The rest follows. 
\end{proof}


%% file: basis.tex
\section{A basis of the walled Brauer algebra}\label{section:basis}
In Theorem~\ref{thm:murphy}, a basis of the symmetric group algebra was
indexed by pairs of tableaux. These tableaux can be identified with
paths of partitions. The other way around we will introduce tableaux
replacing paths to $(\lambda,\mu)$ in this
section.

Further, we define a basis of the walled Brauer algebra
indexed by these tableaux. In order
to obtain this we introduce a
generalization of Young diagrams and tableaux.

We call subsets $[\rho]$ of $\mathbb{N}^2$ which can be written as
$[\nu]\setminus[\lambda]$ for compositions $\nu$, $\lambda$ \emph{skew
  diagrams}. Further we extend the term tableau to every
injective map from a skew diagram $[\rho]$ into the integers and call
such a tableau a $\rho$-tableau (here $\rho$ is just a notation and
has not the meaning of a sequence of certain numbers). As before these
(skew-)tableaux get depicted by an array of boxes and the
corresponding numbers in these boxes, see Figure~\ref{figure:skewdiagram}.  
Given a tableau $\tt$ we call the image of $\tt$ the filling of the
tableau and denote it by $\fil(\tt)$.

In an even wider context, $\rho$-tableaux can be defined as bijections
from $[\rho]$ into a linearly ordered set. Then the terms row
standard, standard and the dominance order on tableaux can be defined
accordingly. 

In particular, we consider maps to subsets of  $\mathbb{Z}$ ordered by $\leq$ or
$\geq$. If the subset of  $\mathbb{Z}$ is ordered by $\leq$, then row standard,
etc.~is defined as before. If the subset of  $\mathbb{Z}$ is ordered by $\geq$ then
row standard means that in each row the entries are decreasing from
left to right. To avoid confusion we will call such tableaux anti-row
standard and similarly a tableau is  anti-standard if it is anti-row
standard and the entries in each column are decreasing from top to
bottom. 

The dominance order $\trianglerighteq^{\op} $ on the set of
anti-row standard tableaux then can be 
explained as follows: If $\ss$ and $\tt$ are anti-row standard tableaux
then  $\tt \trianglerighteq^{\op} \ss$ if
and only if $\shape(\tt\uparrow i) \trianglerighteq^{\op}
\shape(\ss\uparrow i)$, where $\tt\uparrow i$ denotes the restriction
of the tableau to all integers greater or equal to $i$.

If $\tt$ is a tableau, let $d(\tt)$ be the permutation such that in
$\tt. d(\tt)$ the entries increase from left to right row by row and
let $d^{\op}(\tt)$ be the permutation such that in $\tt. d^{\op}(\tt)$
the entries decrease from left to right row by row.

\begin{figure}[h!]
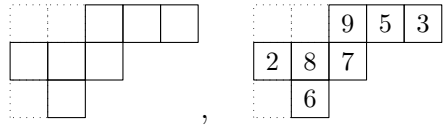

  \centering
 \epsfbox{figures/youngtableaux.3}, \quad\epsfbox{figures/youngtableaux.2}
  \caption{A \emph{skew diagram} $[\rho]$ and a $\rho$-tableau}
  \label{figure:skewdiagram} 
\end{figure}

Recall, that $r$ and $s$ are two fixed non-negative integers.
Let $(\tt, \uu, \vv)$ be a triple of tableaux with the following properties:
\begin{itemize}
\item $\tt$ is a row-standard $\nu$-tableau 
  with entries $\{1,2,\dots,r\}$ where $\nu$ is  a partition of $r$,
\item $\uu$ is a row anti-standard $\rho$-tableau with $k$ boxes, such that
  $[\rho]\subseteq[\nu]$; $[\nu]\setminus [\rho]=[\lambda]$ where $\lambda$ is a composition of
  $r-k$, 
\item $\vv$ is a row-standard $\mu$-tableau with $\mu$ a partition of
  $s-k$, 
\item the entries of $\uu$ and $\vv$ are $\{ 1,2,\dots,s \}$.
\end{itemize}
We call such a triple $(\tt,\uu,\vv)$ a row-standard triple, we call
$(\lambda, \mu)$ its shape and $r+s$ its length.  
We call $(\tt, \uu, \vv)$ standard if  additionally:
\begin{itemize}
\item $\tt$ is standard,
\item $\lambda$ is a partition of $r-k$ and $\uu$ is anti-standard,
\item $\vv$ is standard.
\end{itemize}
%
%
We denote the set of all pairs $(\lambda, \mu)$ of partitions of $r-k$
and $s-k$ resp.~for some non-negative $k$ by $\Lambda(r,s)$ and the set of all
standard triples of shape $(\lambda, \mu)$ by $M(\lambda, \mu)$. 


\begin{prop}
  Let $(\lambda, \mu)$ be a pair of partitions in
  $\Lambda(r,s)$. There is a bijection 
  between the set $M(\lambda, \mu)$ and the set of paths to $(\lambda,
  \mu)$.  
\end{prop}
\begin{proof}
  To each path $Y=(\emptyset\to Y_1\to Y_2 \to\dots\to
  Y_{r+s})$ we relate an element of $M(\lambda, \mu)$ in the following way: 
  \begin{itemize}
  \item
    Let $\nu=\lambda^{(r)}$, $\lambda=\lambda^{(r+s)}$, 
    $\mu=\mu^{(r+s)}$ and $\rho$ be defined
    by $[\rho]=[\nu]\backslash[\lambda]$. 
  \item Let $\tt$, $\uu$ and $\vv$ be the $\nu$-,
    $\rho$- and
    $\mu$-tableaux respectively satisfying: 
    \begin{itemize} 
    \item For $1\leq i\leq r$ if $[\lambda^{(i)}]$ is obtained from
      $[\lambda^{(i-1)}]$ by adding 
      a box then $i$ is the entry of this box inside $\tt$. 
    \item For $1\leq i\leq s$ if $[\mu^{(r+i)}]$ is obtained from
      $[\mu^{(r+i-1)}]$ by adding
      a box  then $i$ is the entry of this box  inside $\vv$.  
    \item For $1\leq i\leq s$ if $[\lambda^{(r+i)}]$ is obtained from
      $[\lambda^{(r+i-1)}]$ by deleting 
      a box  then $i$ is the entry of this box inside $\uu$. 
      
    \end{itemize}
  \end{itemize}
  Obviously this is an $1$-$1$-correspondence.
\end{proof}

\begin{exmp}
  Let $Y$ be the path to 
 $\tuple{youngtableaux.6}{youngtableaux.9}{.5}$ 
with
\begin{align*}
Y&=\tuple{0}{0}{.5}\to \tuple{youngtableaux.4}{0}{.5}
\to \tuple{youngtableaux.6}{0}{.5}\to \tuple{youngtableaux.7}{0}{.5}
\to \tuple{youngtableaux.8}{0}{.5}\\
\to & \tuple{youngtableaux.7}{0}{.5}
\to \tuple{youngtableaux.7}{youngtableaux.4}{.5}
\to \tuple{youngtableaux.7}{youngtableaux.5}{.5}
\to \tuple{youngtableaux.6}{youngtableaux.5}{.5}
\to \tuple{youngtableaux.6}{youngtableaux.9}{.5}
\end{align*}
  \Yautoscale1 
  If we go through the procedure of the proof, 
  then  this path is related to the following standard triple $(\tt,\uu,\vv)$
  of tableaux
  \[
  (\tt,\uu,\vv) = 
  \triple{youngtableaux.10}{youngtableaux.12}{youngtableaux.11}{1}{2}.
  \]
\end{exmp}
For each row-standard triple of
tableaux $(\tt,\uu,\vv)$ we will 
define a linear combination of generalized diagrams. Note that in general,  these
linear combinations are not elements of the walled Brauer algebra
since they involve different numbers of vertices in the top and in the
bottom row. Eventually, we will use these linear combinations to
define a basis  for the walled Brauer algebra. 

We first define two additional
tableaux $\oo$ and $\ss$. The first tableau $\oo$ is the $\nu$-tableau
with the entries of $\uu$ in the same positions as in $\uu$. Into the
remaining boxes we fill in 
different integers  greater than $s$  such that these integers decrease 
in order from left to right along the rows. The particular choice of the
integers does not play a role.
The second tableau
$\ss$ is row standard of shape $(k,s-k)$, such that the entries of the
first row are the entries of $\uu$ and the entries of the second row
are the entries of $\vv$. 

Let  $m_{(\tt,\uu,\vv)}$ be the element  defined via the diagram in
figure \ref{mtuv}. Note that there are $k$ horizontal edges, $r-k$
down- and $s-k$ up-arrows in the top row and $r$ down- and $s$
up-arrows in the bottom row. 
\begin{figure}[h!]
  \centering
  \[ m_{(\tt,\uu,\vv)} =\ \wtangle{diagrams.1}{7cm} \]
  \caption{The element $m_{(\tt,\uu,\vv)}$}
  \label{mtuv}
\end{figure}

If the vertices in the bottom row of this diagram are labeled from
left to right by $1,\ldots,r$ and $1,\ldots,s$ respectively and the
$r-k$ vertices on the upper left are labeled by the additional entries
in $\oo$, decreasing from left to right, then the element
$m_{(\tt,\uu,\vv)}$ can be easily illustrated: the entries of $\tt$,
read row by row, indicate the ending  vertices in the bottom left
of the edges coming from
$y_\nu$. In the same way, the entries of $\oo$ indicate the starting
vertices of the edges ending in $y_\nu$ and the entries of $\vv$
indicate the starting vertices of the edges ending in $y_\mu$.

\begin{exmp}\label{exmtuv}
  Consider the standard triple
  \[
  (\tt,\uu,\vv) =
  \triple{youngtableaux.13}{youngtableaux.14}{youngtableaux.15}{1}{3}.
  \]
  Then $\nu=(4,3,2)$, $\lambda=(2,2,1)$ and $\mu=(3,1,1)$. 
  The additional tableaux are $\oo =
  \wtangle{youngtableaux.16}{2cm}$ and
  $\ss=\wtangle{youngtableaux.17}{2.5cm}$ 
  and
  we get the 
  diagram in Figure~\ref{diagtoexmtuv}. 

  \begin{figure}[h!]
    \centering
    \htangle{diagrams.2}{4cm}
    \caption{$m_{(\tt,\uu,\vv)}$ as in Example \ref{exmtuv}}\label{diagtoexmtuv}
  \end{figure}
\end{exmp}

For  $(\lambda,\mu) \in\Lambda(r,s)$ let $m_{\lambda,\mu} $ be  the element 
\[ m_{\lambda,\mu} := \wtangle{diagrams.3}{4cm}. \]
Let $\tau = (\lambda_1, 1^{(\nu_1-\lambda_1)},\lambda_2,
1^{(\nu_2-\lambda_2)},\dots)$ which is a composition of $r$, then $y_\nu =
y_\tau \cdot z$ where $z$ is an alternating sum of right coset
representatives of $\mathfrak S_\tau$ in $\mathfrak S_\nu$. Note that 
$d^{\op}(\oo)^*y_\tau=y_\lambda d^{\op}(\oo)^*$.
It follows that there is some diagram  $b_{(\tt,\uu,\vv)}$ such that
  $m_{(\tt,\uu,\vv)} =
m_{\lambda,\mu} \cdot b_{(\tt,\uu,\vv)}$. 

We define 
\begin{equation*}
m_{(\tt',\uu',\vv'),(\tt,\uu,\vv)}
:= b_{(\tt',\uu',\vv')}^\ast m_{\lambda,\mu} b_{(\tt,\uu,\vv)}
\end{equation*}
which is an element of $B_{r,s}(x)$. Note that this definition does
not depend on the  choice of $b_{(\tt,\uu,\vv)}$.

\begin{remark}
  Theorem~\ref{thm:murphy} and Proposition~\ref{prop:murphy} still hold
  if $\Std(\lambda)$ is replaced by the set of standard tableaux with
  entries in a fixed ordered set.
  In particular,  $\Std(\lambda)$ can be replaced
  by the set of anti-standard tableaux with a fixed filling. Then the
  basis elements have to be defined by $m^\lambda_{\ss,\tt}={d^{\op}(\ss)}^*y_\lambda
  d^{\op}(\tt)$. One can even allow different sets of
  $\lambda$-tableaux for $\ss$ and $\tt$. 
  
  Suppose now that $(\tt,\uu,\vv)$ and $(\tt',\uu',\vv')$ are
  row-standard triples of shape $(\lambda,\mu)$. 
  Let $\oo$ and $\ss$ be the
  tableaux defined
  above for the triple $(\tt,\uu,\vv)$ and similarly $\oo'$ and
  $\ss'$  for the  triple $(\tt',\uu',\vv')$.  Let
  $m_{\oo,\tt}^\nu={d^{\op}(\oo)}^*y_\nu d(\tt)$ and similarly  
  $m_{\oo',\tt'}^\nu$ and 
  $m_{\vv',\vv}^\mu$ be defined in
  the above sense. Choose $b_{\oo,\tt}^\nu$ and 
  $b_{\oo',\tt'}^{\nu'}$ such that $m_{\oo,\tt}^\nu=y_\lambda
  b_{\oo,\tt}^\nu$ and  $m_{\oo',\tt'}^{\nu'}=y_\lambda
  b_{\oo',\tt'}^{\nu'}$. Then  $m_{\tt',\oo'}^{\nu'}=
  \left({b_{\oo',\tt'}^{\nu'}}^*\right) y_\lambda$ and the element
  $m_{(\tt',\uu',\vv'),(\tt,\uu,\vv)}$
  is given in
  Figure~\ref{figure:basiselement}.

  \begin{figure}[h!]
    \[\wtangle{diagrams.6}{4.2cm}=\wtangle{diagrams.8}{4.2cm}
    =\wtangle{diagrams.5}{4.2cm}\]
    \caption{ $m_{(\tt',\uu',\vv'),(\tt,\uu,\vv)}$}
    \label{figure:basiselement}
  \end{figure} 
  
\end{remark}


%% file: generators.tex
We will show that some of these elements form a cellular basis. The partially
ordered set is $ \Lambda(r,s)$ with the following ordering: 
Let $(\lambda,\mu),(\lambda',\mu')\in
\Lambda(r,s)$. Then $\lambda\vdash (r-k)$,    $\mu\vdash (s-k)$,
$\lambda'\vdash (r-k')$ and     $\mu'\vdash (s-k')$ for some
$k$ and $k'$. We define
\begin{equation*}
  (\lambda',\mu')\unrhd (\lambda,\mu):\Leftrightarrow
  k'>k\text{ or }
  (k'=k,\lambda'\unrhd \lambda \text{ and }\mu'\unrhd\mu).
\end{equation*}

For
$(\lambda,\mu)\in\Lambda(r,s)$
let $\widetilde{m}_{\lambda,\mu}\in B_{r,s}(x)$ be
the element shown in Figure~\ref{figure:mlambdamu}. 
\begin{figure}[h!]
  \[\widetilde{m}_{\lambda,\mu}=\wtangle{diagrams.4}{6cm}\]
  \caption{$\widetilde{m}_{\lambda,\mu}$}\label{figure:mlambdamu}
\end{figure}  
This element coincides with $m_{\lambda,\mu}$ except for $k$
additional horizontal edges on the top and on the bottom. 

 Further,
let  $B^{\trianglerighteq(\lambda,\mu)}$ be the 
ideal of $B_{r,s}(x)$
generated by all $\widetilde{m}_{\lambda',\mu'}$ with 
$(\lambda',\mu')\unrhd (\lambda,\mu)$. We claim that as an $R$-module, 
$B^{\trianglerighteq(\lambda,\mu)}$ is
generated by all   $m_{(\tt',\uu',\vv'),(\tt,\uu,\vv)}$
where $(\tt',\uu',\vv')$ and
$(\tt,\uu,\vv)$ are  row-standard
triples of shape $(\lambda',\mu')$ with
$(\lambda',\mu')\unrhd(\lambda,\mu)$.
By definition, these elements are elements of $B^{\trianglerighteq(\lambda,\mu)}$. On
the other hand, each element of $B^{\trianglerighteq(\lambda,\mu)}$ is a linear
combination of $b_1\widetilde{m}_{\lambda',\mu'}b_2$ with
$(\lambda',\mu')\unrhd (\lambda,\mu)$ and $b_1,b_2$  walled Brauer
diagrams. By deleting cycles, such an element can be written as (possibly a multiple of)
$b_1'm_{\lambda',\mu'}b_2'$
where $b_1'$ and $b_2'$ are generalized diagrams. If the  diagram
$b_1'm_{\lambda',\mu'}b_2'$ has more horizontal edges
than $ \widetilde{m}_{\lambda',\mu'}$, say $k''$ edges, then $b_1'm_{\lambda',\mu'}b_2'$
can be written as  a linear combination of  $m_{(\tt',\uu',\vv'),(\tt,\uu,\vv)}$ with 
$(\tt',\uu',\vv')$ and $(\tt,\uu,\vv)$ row standard of shape
$((1^{r-k''}),(1^{s-k''}))\rhd(\lambda',\mu')$. 
If the number of
horizontal edges is the same in  $ \widetilde{m}_{\lambda',\mu'}$ and
$b_1\widetilde{m}_{\lambda',\mu'}b_2$, then 
$b_1'm_{\lambda',\mu'}b_2'=\pm m_{(\tt',\uu',\vv'),(\tt,\uu,\vv)}$ with 
$(\tt',\uu',\vv')$ and $(\tt,\uu,\vv)$ row standard 
of shape $(\lambda',\mu')$.

We call the elements of
$B^{\trianglerighteq(\lambda,\mu)}$ \emph{$(\lambda,\mu)$-elements}.
Note that $B^{\trianglerighteq((1^{r-k}),(1^{s-k}))}$ is the ideal of the walled
Brauer algebra generated by the walled Brauer diagrams with at least
$k$ horizontal edges on top and on bottom.

Suppose now that
$ (\tt,\uu,\vv)$ and  $(\tilde\tt,\tilde\uu,\tilde\vv)$  are
row-standard triples of shape  $(\lambda,\mu)$ and 
$(\tilde\lambda,\tilde\mu)$ respectively. Let $\oo$ be defined as
above and let $\tilde\oo$ be the corresponding tableau for
$(\tilde\tt,\tilde\uu,\tilde\vv)$ such that $\oo$ and $\tilde\oo$ have the
same additional entries. 
Then we define
$  (\tilde\tt,\tilde\uu,\tilde\vv)\unrhd (\tt,\uu,\vv)$ 
if and only  if
  \begin{itemize}
  \item $(\tilde\lambda,\tilde\mu)\rhd (\lambda,\mu)$ or
  \item $(\tilde\lambda,\tilde\mu)= (\lambda,\mu)$, $\tilde\tt\unrhd
    \tt$, 
    $\tilde\oo\unrhd^{\op} \oo$ and $\tilde\vv\unrhd\vv$ .
  \end{itemize}

\begin{lemma}\label{lemma:straightening}
  Let $(\tt',\uu',\vv')$ and
  $(\tt,\uu,\vv)$ be row-standard triples of shape
  $(\lambda,\mu)$ such that $(\tt,\uu,\vv)$ is
  not standard. Then $m_{(\tt',\uu',\vv'),(\tt,\uu,\vv)}$ is a linear
  combination of elements
  $m_{(\tt',\uu',\vv'),(\tilde\tt,\tilde\uu,\tilde\vv)}$ with $
  (\tilde\tt,\tilde\uu,\tilde\vv)\rhd (\tt,\uu,\vv)$ and of
  $(\lambda',\mu')$-elements such that $(\lambda',\mu')\rhd(\lambda,\mu)$.  
\end{lemma}

\begin{proof}
  Recall that $m_{(\tt',\uu',\vv'),(\tt,\uu,\vv)}$ can be
  depicted as in Figure~\ref{figure:basiselement}.
  Suppose first that $\vv$ is not standard.
  By Proposition~\ref{prop:murphy},  $m_{\vv',\vv}^\mu$ is a linear
  combination of 
 $m_{\tilde\vv',\tilde\vv}^{\tilde\mu}$ where either
  $\tilde\vv'=\vv'$ and   $\tilde\vv\rhd  \vv$ (and thus $\tilde\mu=\mu$)
  or $\tilde\mu\rhd \mu$. Replacing 
  $m_{\vv',\vv}^\mu$ in
  $m_{(\tt',\uu',\vv'),(\tt,\uu,\vv)}$ by this
  linear combination shows the result.

  Now suppose that $\tt$ is not standard or $\oo$  is
  not anti-standard. Again by  
  Proposition~\ref{prop:murphy}, $m_{\oo,\tt}^\nu$ is a
  linear combination of $m_{\tilde\oo,\tilde\tt}^{\tilde\nu}$ with
  $\tilde\oo$ anti-standard and $\tilde\tt$ standard,
  $\tilde\oo\unrhd^{\op} \oo$ and $\tilde\tt\unrhd\tt$ and 
  $m_{(\tt',\uu',\vv'),(\tt,\uu,\vv)}$ can be accordingly written as a
  linear combination. Note that $\lambda$ is by definition the shape of
  $\oo\uparrow s+1$.

  If $\shape(\tilde\oo\uparrow
  s+1)\rhd\shape(\oo\uparrow s+1)$, then the element obtained by replacing
  $m_{\oo,\tt}^\nu$  by  $m_{\tilde\oo,\tilde\tt}^{\tilde\nu}$ in 
  $m_{(\tt',\uu',\vv'),(\tt,\uu,\vv)}$ is a
  $(\tilde{\lambda},\mu)$-element where
  $\tilde{\lambda}=\shape(\tilde\oo\uparrow s+1)\rhd\lambda$.

  If $\shape(\tilde\oo\uparrow s+1)=\shape(\oo\uparrow s+1)=\lambda$
  then we
  have $\tilde\oo\uparrow s+1=\oo\uparrow s+1$, since $\oo\uparrow s+1$ is the
  maximal row anti-standard $\lambda$-tableau with this filling.
  But then the
  element obtained by replacing $m_{\oo,\tt}^\nu$ by  $m_{\tilde\oo,\tilde\tt}^{\tilde\nu}$ is 
  $m_{(\tt',\uu',\vv'),(\tilde\tt,\tilde\uu,\vv)}$ where $\tilde\uu$ is
  obtained by restriction of $\tilde\oo$.

\end{proof}
\begin{remark}\label{remark:straightening}
  The proof of Lemma~\ref{lemma:straightening} shows: if
  $m_{(\tt',\uu',\vv'),(\tilde\tt,\tilde\uu,\tilde\vv)}$ appears in
  the linear combination in Lemma~\ref{lemma:straightening} and
  $\tilde\ss$ is defined as above for
  $(\tilde\tt,\tilde\uu,\tilde\vv)$,
  then $\tilde\ss=\ss$. 

\end{remark}

\begin{thm}\label{thm:basis}
  The set 
  \[
  \{ m_{(\tt',\uu',\vv'),(\tt,\uu,\vv)}\mid
  (\tt',\uu',\vv'),(\tt,\uu,\vv)\in M(\lambda,\mu),(\lambda,\mu)\in
  \Lambda(r,s)\}\] 
  is a cellular basis of the walled Brauer algebra $B_{r,s}(x)$. The
  partial order on $\Lambda(r,s)$ is given by $\unrhd$.
\end{thm}

\begin{proof}
  Note that an analogue version of Lemma~\ref{lemma:straightening} for
  $(\tt',\uu',\vv')$  not standard is
  also valid. Thus
  it follows from Lemma~\ref{lemma:straightening}
  that  the $m_{(\tt',\uu',\vv'),(\tt,\uu,\vv)}$ for
  standard triples $(\tt',\uu',\vv'),(\tt,\uu,\vv)$
  of shape $(\lambda',\mu')\unrhd(\lambda,\mu)$
  generate $ B^{\trianglerighteq(\lambda,\mu)}$. In particular, the set of all such
  elements    generates $B_{r,s}(x)=B^{\trianglerighteq((1^r),(1^s))}$.  
  Since the cardinality of the set of these elements is equal to the
  rank of $B_{r,s}(x)$, this set is  linearly independent. 
  
  Recall the anti-involution $*$ of $B_{r,s}(x)$ given by reflecting
  diagrams at a
  horizontal axis.  Then by definition,
  $m_{(\tt',\uu',\vv'),(\tt,\uu,\vv)}^*=
  m_{(\tt,\uu,\vv),(\tt',\uu',\vv')}$.

  It remains to prove the second
  property of a cellular algebra. Let $b\in B_{r,s}(x)$ and
  $(\tt,\uu,\vv)$ and $(\tt',\uu',\vv')$ be  standard triples
  of shape $(\lambda,\mu)$.
  Then $m_{(\tt',\uu',\vv'),(\tt,\uu,\vv)}b$ is an element of
  $B^{\trianglerighteq(\lambda,\mu)}$  and thus can be written as a linear combination
  of
  $m_{(\tilde\tt',\tilde\uu',\tilde\vv'),(\tilde\tt,\tilde\uu,\tilde\vv)}$
  for certain standard triples of shape dominating
  $(\lambda,\mu)$ by  Lemma~\ref{lemma:straightening}.
  If $m_{(\tilde\tt',\tilde\uu',\tilde\vv'),(\tilde\tt,\tilde\uu,\tilde\vv)}$
  appears in this linear combination, then we have by construction that 
  $(\tilde\tt',\tilde\uu',\tilde\vv')= (\tt',\uu',\vv')$ or the shape 
  of $(\tilde\tt',\tilde\uu',\tilde\vv')$ strictly dominates  $(\lambda,\mu)$.

  Let $(\tt_0,\uu_0,\vv_0)$ be the unique standard triple of shape
  $(\lambda,\mu)$ such that 
  $\nu_0=(\lambda_1,\ldots,\lambda_l,1^k)$ is the shape of $\tt_0$ and
   $d(\tt_0)$, $d(\ss_0)$, $d(\vv_0)$ and
  $d^{\op}(\oo_0)$ are the identity. 

  For each standard triple  $(\tt',\uu',\vv')$
  there is a (linear combination of) walled Brauer diagram(s) $d$ without
  horizontal edges such that
  $dm_{(\tt_0,\uu_0,\vv_0),(\overline\tt,\overline\uu,\overline\vv)}=
  m_{(\tt',\uu',\vv'),(\overline\tt,\overline\uu,\overline\vv)}$
  for all standard triples  $(\overline\tt,\overline\uu,\overline\vv)$
  and the result follows. 

\end{proof}

\begin{remark}
\begin{enumerate}
\item The ideal $B^{\trianglerighteq(\lambda,\mu)}$  is exactly the cell ideal $B_{r,s}(x)(\trianglerighteq(\lambda,\mu))$. 
\item
  If $r=0$ or $s=0$, then  $B_{r,s}(x)$ is the group algebra of a
  symmetric group and  the basis of $B_{r,s}(x)$ we just defined
  coincides with the basis from Theorem~\ref{thm:murphy}. 
\end{enumerate}
\end{remark}


%% file: restriction.tex
\section{Restriction}\label{section:restriction}
In this section we consider the restriction of cell modules to
$B_{r,s-1}(x)$ or $B_{r-1,0}(x)$ for $s=0$. 
\begin{defn}
  For every $(\lambda,\mu)\in\Lambda(r,s)$ let
  the cell module $C^{(\lambda,\mu)}$ be the right  $B_{r,s}(x)$-submodule of
  $B^{\trianglerighteq(\lambda,\mu)} /
  B^{\triangleright(\lambda,\mu)}$ generated by
  $\widetilde{m}_{\lambda,\mu} + B^{\triangleright(\lambda,\mu)}$
  where $ B^{\triangleright(\lambda,\mu)}=
  \bigcap_{(\lambda',\mu')\triangleright(\lambda,\mu)} B^{\trianglerighteq(\lambda',\mu')}$. 
  These modules  in fact coincide with the cell modules in
  \cite{grahamlehrer}. 

  A basis of $C^{(\lambda,\mu)}$ is indexed by $M(\lambda,\mu)$ and in
  abuse of notation we use the set
  \[
  \{ m_{(\tt,\uu,\vv)} \ |\ (\tt,\uu,\vv)\in M(\lambda,\mu) \}
  \]
  as a basis for $C^{(\lambda,\mu)}$.
\end{defn}
Note that the action of the walled Brauer algebra on a cell module
is again given by
concatenation, deleting cycles by multiplication with $x$ and  factoring out diagrams involving $m_{\lambda',\mu'}$
with $(\lambda',\mu')\triangleright(\lambda,\mu)$.

If $s>0$, then the walled Brauer diagrams of $B_{r,s-1}(x)$ can be embedded into
$B_{r,s}(x)$ by adding  a vertical edge
connecting the rightmost vertex in the top row with the rightmost
vertex in the bottom row. Similarly, $B_{r-1,0}(x)$ can be embedded
into $B_{r,0}(x)$. We get a tower of algebras 
\[
R=B_{1,0}(x) \subset B_{2,0}(x) \subset \dots \subset B_{r,0}(x)
\subset B_{r,1}(x) \subset \dots \subset B_{r,s}(x).
\]
In this section, we describe the behavior of the cell modules under
restriction in this tower of algebras.  Fix $r$ and $s$ and let
$\tilde{B}= B_{r-1,0}(x)$ if $s=0$ and $\tilde{B}=B_{r,s-1}(x)$ if
$s\geq 1$. 
For a $B_{r,s}(x)$-module $M$ we denote with $\Res(M)$ its restriction
to the subalgebra $\tilde{B}$.

Let $(\tt,\uu,\vv)$ be a standard triple in $M(\lambda,\mu)$ for
$(\lambda,\mu)\in\Lambda(r,s)$. First, suppose $s=0$. Then $\uu$ and $\vv$ are empty
tableaux. Let $\tt'=\tt\downarrow r-1$, then
$(\tt',\emptyset,\emptyset)$ is a standard triple of shape
$(\lambda',\mu)\in \Lambda(r-1,0)$. We call this triple 
the restriction of the triple
$(\tt,\uu,\vv)$ and denote it 
by $\Res(\tt,\uu,\vv)$. 

If $s\geq 1$ then we  set $\uu' :=
\uu\downarrow (s-1)$ and $\vv' := \vv\downarrow (s-1)$. The triple
$(\tt,\uu',\vv')$ is in $M(\lambda',\mu')$ for
$(\lambda',\mu')\in\Lambda(r,s-1)$. More precisely $(\tt,\uu',\vv')$
is either an element of $M(\lambda,\mu')$ and $\mu'$ is obtained from
$[\mu]$ by removing a box, or $(\tt,\uu',\vv')$ is an element of
$M(\lambda',\mu)$ and $[\lambda']$ is obtained from $[\lambda]$ by
adding a box. 
Again, we call $(\tt,\uu',\vv')$ the restriction of the triple
$(\tt,\uu,\vv)$ and denote it 
by $\Res(\tt,\uu,\vv)$. 

For $(\lambda,\mu)\in\Lambda(r,s)$ we define the set
$\Res(\lambda,\mu)$ to be the set of tuples $(\lambda',\mu')$ occurring
as shapes of restrictions of standard triples of shape $(\lambda,\mu)$. 
So if $s=0$, then $\Res(\lambda,\emptyset)=\{(\lambda',\emptyset)\}$
where $[\lambda']$ is obtained from $[\lambda]$ by removing a box. 
If $s\geq 1$, then $\Res(\lambda,\mu)$ is the set of tuples
$(\lambda',\mu')$  obtained from $(\lambda,\mu)$ by
either removing a box from $[\mu]$ or adding a box to $[\lambda]$ if
$[\lambda]$ has less then $r$ boxes. 

We are now able to define for each $(\lambda',\mu')\in\Res(\lambda,\mu)$ the
following $R$-submodules of $C^{(\lambda,\mu)}$  
\[
U^{\trianglerighteq(\lambda',\mu')} := \langle m_{(\tt,\uu,\vv)}\ |\
\shape(\Res(\tt,\uu,\vv)) \trianglerighteq
(\lambda',\mu')\rangle_{R-\operatorname{mod}},
\] 
\[
U^{\triangleright(\lambda',\mu')} := \langle m_{(\tt,\uu,\vv)}\ |\
\shape(\Res(\tt,\uu,\vv)) \rhd
(\lambda',\mu')\rangle_{R-\operatorname{mod}}. 
\]

If $s=0$ then $\mu$ is the empty partition and $\uu$ and $\vv$ are
empty tableaux which can be omitted, e.~g.~$ U^{\trianglerighteq(\lambda',\emptyset)}=U^{\trianglerighteq\lambda'}$. We have the classical result for
the group algebra of the symmetric group $R\mathfrak{S}_r$: 
\begin{thm}[\cite{mathas}]\label{thm:restriction_symmetric_group}

  The modules $U^{\trianglerighteq\lambda'}$ and 
  $U^{\triangleright \lambda'}$ are $R\mathfrak{S}_{r-1}$-submodules of
  $\Res C^\lambda$. The modules $U^{\trianglerighteq\lambda'}$ are
  linearly ordered by inclusion and give a filtration of the restricted
  cell module $\Res C^\lambda$ with factors
  $U^{\trianglerighteq\lambda'}/U^{\triangleright \lambda'}\cong
  C^{\lambda'}$. The isomorphism is given by $m_\tt+U^{\triangleright
    \lambda'} \mapsto m_{\Res \tt}$.

\end{thm}

As the main result of this section we obtain a cell filtration of the
restriction of cell modules. 
\begin{thm}\label{thm:restriction}
  \begin{enumerate}
  \item The modules $U^{\trianglerighteq(\lambda',\mu')}$ and
    $U^{\triangleright(\lambda',\mu')}$ are $\tilde{B}$-submodules
    of $\Res C^{(\lambda,\mu)}$.
  \item  The set $\Res(\lambda,\mu)$ is ordered linearly 
    by dominance order. Thus the modules 
    $U^{\trianglerighteq(\lambda',\mu')}$ with
    $(\lambda',\mu')\in\Res(\lambda,\mu)$ are accordingly ordered by
    inclusion and we get a
    filtration  
    \[
    \{0\}\subset U^{\trianglerighteq(\lambda'_1,\mu'_1)} \subset U^{\trianglerighteq(\lambda'_2,\mu'_2)} \subset \dots \subset U^{\trianglerighteq(\lambda'_l,\mu'_l)} 
    = \Res C^{(\lambda,\mu)} 
    \]
    of $\Res C^{(\lambda,\mu)}$ by $\tilde{B}$-modules.
  \item Further a basis of each quotient $U^{\trianglerighteq(\lambda',\mu')} /
    U^{\triangleright(\lambda',\mu')}$ is given by
    \[\{ m_{(\tt,\uu,\vv)} +
    U^{\triangleright(\lambda',\mu')} \ |\
    \shape(\Res(\tt,\uu,\vv)) = (\lambda',\mu') \}.\]  
  \item
    The $R$-linear map $U^{\trianglerighteq(\lambda',\mu')} /
    U^{\triangleright(\lambda',\mu')}\to C^{(\lambda',\mu')}$ given by 
    \[
    m_{(\tt,\uu,\vv)} + U^{\triangleright(\lambda',\mu')}
    \mapsto m_{\Res(\tt,\uu,\vv)}
    \]
    is an isomorphism of  
    $\tilde{B}$-modules for
    every $(\lambda',\mu')\in\Res(\lambda,\mu)$.
  \end{enumerate}
\end{thm}
\begin{proof}
  It can be seen easily, that $\Res(\lambda,\mu)$ is ordered
  linearly. 
  If $s=0$ then the theorem follows from
  Theorem~\ref{thm:restriction_symmetric_group}.
  So  suppose $s\geq 1$.  Let $(\tt,\uu,\vv)\in M(\lambda,\mu)$
  and $b\in \tilde{B}=B_{r,s-1}(x)$. 
  Suppose, the elements
  $\alpha_{(\tilde{\tt},\tilde{\uu},\tilde{\vv})}\in R$ are such that
  \begin{equation}\label{eqn:mult}
    m_{(\tt,\uu,\vv)}\cdot b 
    = \sum_{(\tilde{\tt},\tilde{\uu},\tilde{\vv})\in M(\lambda,\mu)}
    {\alpha_{(\tilde{\tt},\tilde{\uu},\tilde{\vv})}
      m_{(\tilde{\tt},\tilde{\uu},\tilde{\vv})}}.
  \end{equation}
  in $C^{(\lambda,\mu)}$. We have to show that
  $\alpha_{(\tilde{\tt},\tilde{\uu},\tilde{\vv})}\neq 0$ implies that  
  $\shape(\Res(\tilde{\tt},\tilde{\uu},\tilde{\vv}))\trianglerighteq
  \shape(\Res(\tt,\uu,\vv))=(\lambda',\mu')$ and that the following holds in
  $C^{(\lambda',\mu')}$:

\begin{equation}\label{eqn:multres}
    m_{\Res(\tt,\uu,\vv)}\cdot b 
    = \sum_{\substack{(\tilde{\tt},\tilde{\uu},\tilde{\vv})\in M(\lambda,\mu)\\
        \Res(\tilde{\tt},\tilde{\uu},\tilde{\vv})=(\lambda',\mu')}}
    {\alpha_{(\tilde{\tt},\tilde{\uu},\tilde{\vv})}
      m_{\Res(\tilde{\tt},\tilde{\uu},\tilde{\vv})}}.
  \end{equation}

  Clearly, it is enough to show this claim for generators $b$ of $
  B_{r,s-1}(x)$. It is well known that the walled Brauer algebra is
  generated by walled diagrams of the form
  $\htangle{generators.1}{.3cm}$, 
  $\htangle{generators.2}{.3cm}$ and  
  $\htangle{generators.3}{.3cm}$.
  Recall the definition of the element  $m_{(\tt,\uu,\vv)}$ in
  Figure~\ref{mtuv}. 

  If $b=\htangle{generators.1}{.3cm}\in B_{r,s-1}(x)$ is a basic
  transposition on the `left side' of 
  the wall, then 
  \[
  m_{(\tt,\uu,\vv)}\cdot b =  m_{(\tt',\uu,\vv)},
  \]
  where   $\tt'$ is obtained from $\tt$ by interchanging two entries. 
  Let $\nu=\shape(\tt)$. If
  $m_{\tt'}=\sum_{\tilde\tt}a_{\tilde\tt}m_{\tilde\tt}$ in $C^\nu$
  then by the construction in the proof of
  Lemma~\ref{lemma:straightening} we have
  \[
  m_{(\tt',\uu,\vv)}=\sum_{\tilde\tt}a_{\tilde\tt}m_{(\tilde\tt,\uu,\vv)} 
  \]
  in $C^{(\lambda,\mu)}$. So if
  $a_{(\tilde\tt,\uu,\vv)}=a_{\tilde\tt}\neq 0$ then $\shape
  (\tilde\tt)=\shape(\tt)=\nu$ and thus
  $\shape(\Res(\tilde{\tt},\uu,\vv))=
  \shape(\Res({\tt},\uu,\vv))$. By the same  considerations we have 
  \[
  m_{\Res(\tt',\uu,\vv)}=\sum_{\tilde\tt}a_{\tilde\tt}m_{\Res(\tilde\tt,\uu,\vv)} 
  \]
  in $C^{(\lambda',\mu')}$.

  Suppose now that
  $b=\htangle{generators.2}{.3cm}\in 
  B_{r,s-1}(x)$ is a basic transposition on the `right 
  side' of the wall, then 
  \[
  m_{(\tt,\uu,\vv)}\cdot b   =\pm m_{(\tt,\uu',\vv')},
  \]
  where $(\uu',\vv')$ is obtained from $(\uu,\vv)$ by interchanging two entries and reordering the entries in each row. 	
   Keep in mind, that the position of $s$ in the pair
  $(\uu',\vv')$ is the same as it is in $(\uu,\vv)$. 

   Suppose 
  $a_{(\tilde{\tt},\tilde{\uu},\tilde{\vv})}\neq 0$.
  By construction 
  $\tilde{\vv}\trianglerighteq\vv'$ and
  $\tilde\oo\trianglerighteq^{\op}\oo'$, where $\tilde\oo$ and $\oo'$ are
  the additional tableaux to $(\tilde{\tt},\tilde{\uu},\tilde{\vv})$
  and $(\tt',\uu',\vv')$ with the same additional entries. 

  In particular, we have
  $\shape(\tilde{\vv}\downarrow
  (s-1))\trianglerighteq\shape(\vv'\downarrow (s-1))$. 
  Note that $\vv$ and
  $\vv'$ have the same shape and the position of $s$ in $\vv$ and
  $\vv'$ is the same if $s$ appears in $\vv$. Thus	
	we have
  $\shape(\vv'\downarrow (s-1))=\shape(\vv\downarrow (s-1))=\mu'$, so  $\shape(\tilde{\vv}\downarrow
  (s-1))\trianglerighteq \mu'$.
  
  Let $\oo_{\Res}$ be the
  additional tableau for the triple $\Res(\tt,\uu,\vv)$. Similarly,
  let   $\tilde\oo_{\Res}$ and
  $\oo_{\Res}'$ be defined.  Since
  $\tilde\oo\trianglerighteq^{\op}\oo'$ we have 
  $\tilde\oo_{\Res}\trianglerighteq^{\op}\oo'_{\Res}$
  and thus
  $\shape(\tilde\oo_{\Res}\uparrow s)
  \trianglerighteq^{\op}\shape(\oo_{\Res}'\uparrow s)
  =\shape(\oo_{\Res}\uparrow s)=\lambda'$. It follows that
  $\shape(\Res(\tilde{\tt},\tilde{\uu},\tilde{\vv}))
  =(\shape(\tilde\oo_{\Res}\uparrow s),\shape(\tilde{\vv}\downarrow
  (s-1)))
  \trianglerighteq
  \shape(\Res(\tt,\uu,\vv))=(\lambda',\mu')$. 
  Equation~\eqref{eqn:multres} follows by the classical result and the
  proof of Lemma~\ref{lemma:straightening}. 

  Now let $b = e =\htangle{generators.3}{.3cm}\in B_{r,s-1}(x)$. Note
  that in this case, $s\geq 2$. At
  first sight the action of $e$ seems to be more involved than the
  action of the transpositions,  but it can also be described
  combinatorially if we do a case-by-case analysis.  
  We have to take a look at the positions of $r$ in $\tt$ and $1$ in
  $\uu$ or $\vv$ and distinguish three cases. 
  \begin{itemize}
  \item
    First case: $r$ is  an entry in row $l$ of $\tt$, $1$ is  an entry
    in row $l$ of $\uu$.

    Denote the other entries in row $l$ of $\tt$ with $a_1,\dots,a_k$.
    So we have
    \[
    (\tt,\uu,\vv) = \left(\
      \xy<0cm,1.25cm>\xymatrix @R .5cm @C .5cm{
        *={}
        & *={}
        & *={}
        & *={}
        & *={}
        & *={}
        & *={} \ar @{--} [llllll]
        \\
        *={} \ar @{--} [u]
        & *={} \ar @{-} [l]
        & *={} \ar @{-} [l]
        & *={} \ar @{-} [l]
        & *={} \ar @{-} [l]
        & *={} \ar @{-} [l] \ar @{~~} [ru]
        \\
        *={} \ar @{-} [u]
        & *={} \ar @{-} [l] \ar @{-} [u] \ar @{} [ul]|{a_1}
        & *={} \ar @{-} [l] \ar @{-} [u] \ar @{} [ul]|{a_2}
        & *={} \ar @{-} [l] \ar @{-} [u] \ar @{} [ul]|{\dots}
        & *={} \ar @{-} [l] \ar @{-} [u] \ar @{} [ul]|{a_k}
        & *={} \ar @{-} [l] \ar @{-} [u] \ar @{} [ul]|{r}
        \\
        *={}
        \\
        *={}
        \\
        *={} \ar @{--} [uuu]
        & *={} \ar @{--} [l] \ar @{~~} [rrruuu]
      }\endxy
      ,
      \xy<0cm,1.25cm>\xymatrix @R .5cm @C .5cm{
        *={}
        & *={}
        & *={}
        & *={}
        & *={}
        & *={}
        & *={} \ar @{--} [lll]
        \\
        *={}
        & *={}
        & *={}
        & *={}
        & *={}
        & *={} \ar @{-} [l] \ar @{~~} [ru]
        \\
        *={}
        & *={}
        & *={}
        & *={}
        & *={} \ar @{-} [u]
        & *={} \ar @{-} [l] \ar @{-} [u] \ar @{} [ul]|{1}
        \\
        *={} \ar @{~~} [rrruuu]
        \\
        *={}
        \\
        *={} \ar @{--} [uu]
        & *={} \ar @{--} [l] \ar @{~~} [rrruuu]
      }\endxy 
      ,
      \xy<0cm,1.25cm>\xymatrix @R .5cm @C .5cm{
        *={}
        & *={}
        & *={}
        & *={}
        & *={}
        & *={} \ar @{--} [lllll]
        \\
        *={}
        \\
        *={}
        \\
        *={}
        \\
        *={} \ar @{--} [uuuu]
        & *={} \ar @{--} [l] \ar @{~~} [rrrruuuu]
      }\endxy\
    \right).
    \]
    This means that in $m_{(\tt,\uu,\vv)}$, the box containing $y_{(k+1)}$
    in $y_\nu$ is connected to the vertices in the bottom row indexed by
    $r$ on the left and $1$ on the right. It can be verified by direct computation
    that $\htangle{restriction.1}{1cm}
    =(x-k)\cdot\htangle{restriction.2}{1cm} $. 
    Thus we get $m_{(\tt,\uu,\vv)}\cdot e = (x-k) \cdot
    m_{(\tilde\tt,\tilde\uu,\tilde\vv)}$ with
    \[
    (\tilde\tt,\tilde\uu,\tilde\vv) = \left(\
      \xy<0cm,1.5cm>\xymatrix @R .5cm @C .5cm{
        *={}
        & *={}
        & *={}
        & *={}
        & *={}
        & *={}
        & *={} \ar @{--} [llllll]
        \\
        *={} \ar @{--} [u]
        & *={} \ar @{-} [l]
        & *={} \ar @{-} [l]
        & *={} \ar @{-} [l]
        & *={} \ar @{-} [l]
        & *={} \ar @{-} [l] \ar @{~~} [ru]
        \\
        *={} \ar @{-} [u]
        & *={} \ar @{-} [l] \ar @{-} [u] \ar @{} [ul]|{a_1}
        & *={} \ar @{-} [l] \ar @{-} [u] \ar @{} [ul]|{a_2}
        & *={} \ar @{-} [l] \ar @{-} [u] \ar @{} [ul]|{\dots}
        & *={} \ar @{-} [l] \ar @{-} [u] \ar @{} [ul]|{a_k}
        \\
        *={}
        \\
        *={}
        \\
        *={} \ar @{--} [uuu]
        & *={} \ar @{-} [l] \ar @{~~} [rrruuu]
        \\
        *={} \ar @{-} [u]
        & *={} \ar @{-} [l] \ar @{-} [u] \ar @{} [ul]|{r}
      }\endxy
      ,
      \xy<0cm,1.5cm>\xymatrix @R .5cm @C .5cm{
        *={}
        & *={}
        & *={}
        & *={}
        & *={}
        & *={}
        & *={} \ar @{--} [lll]
        \\
        *={}
        & *={}
        & *={}
        & *={}
        & *={}
        & *={} \ar @{-} [l] \ar @{~~} [ru]
        \\
        *={}
        & *={}
        & *={}
        & *={}
        & *={} \ar @{-} [u]
        & *={}
        \\
        *={} \ar @{~~} [rrruuu]
        \\
        *={}
        \\
        *={} \ar @{--} [uu]
        & *={} \ar @{-} [l] \ar @{~~} [rrruuu]
        \\
        *={} \ar @{-} [u]
        & *={} \ar @{-} [l] \ar @{-} [u] \ar @{} [ul]|{1}
      }\endxy 
      ,
      \xy<0cm,1.5cm>\xymatrix @R .5cm @C .5cm{
        *={}
        & *={}
        & *={}
        & *={}
        & *={}
        & *={} \ar @{--} [lllll]
        \\
        *={}
        \\
        *={}
        \\
        *={}
        \\
        *={} \ar @{--} [uuuu]
        & *={} \ar @{--} [l] \ar @{~~} [rrrruuuu]
      }\endxy\
    \right)
    \]
    which is a standard triple. 
    In the same way we get the equation
    $m_{\Res(\tt,\uu,\vv)}\cdot e = (x-k) \cdot
    m_{\Res(\tilde\tt,\tilde\uu,\tilde\vv)}$
    in $C^{(\lambda',\mu')}$.

  \item Second case:  $1$ is an entry in row $l$ of $\uu$ and $r$ is
    not in row $l$ of $\tt$.

    Denote the entries in row $l$ of $\tt$ with
    $a_1,\dots,a_k$, then we have 
    \[
    (\tt,\uu,\vv) = \left(\
      \xy<0cm,1.25cm>\xymatrix @R .5cm @C .5cm{
        *={}
        & *={}
        & *={}
        & *={}
        & *={}
        & *={}
        & *={} \ar @{--} [llllll]
        \\
        *={} \ar @{--} [u]
        & *={} \ar @{-} [l]
        & *={} \ar @{-} [l]
        & *={} \ar @{-} [l]
        & *={} \ar @{-} [l]
        & *={} \ar @{-} [l] \ar @{~~} [ru]
        \\
        *={} \ar @{-} [u]
        & *={} \ar @{-} [l] \ar @{-} [u] \ar @{} [ul]|{a_1}
        & *={} \ar @{-} [l] \ar @{-} [u] \ar @{} [ul]|{a_2}
        & *={} \ar @{-} [l] \ar @{-} [u] \ar @{} [ul]|{\dots}
        & *={} \ar @{-} [l] \ar @{-} [u] \ar @{} [ul]|{\dots}
        & *={} \ar @{-} [l] \ar @{-} [u] \ar @{} [ul]|{a_k}
        \\
        *={}
        & *={}
        & *={}
        & *={} \ar @{-} [l] \ar @{~~} [ru]
        \\
        *={}
        & *={}
        & *={} \ar @{-} [u]
        & *={} \ar @{-} [l] \ar @{-} [u] \ar @{} [ul]|{r}
        \\
        *={} \ar @{--} [uuuuu]
        & *={} \ar @{--} [l] \ar @{~~} [ru]
      }\endxy
      ,
      \xy<0cm,1.25cm>\xymatrix @R .5cm @C .5cm{
        *={}
        & *={}
        & *={}
        & *={}
        & *={}
        & *={}
        & *={} \ar @{--} [lll]
        \\
        *={}
        & *={}
        & *={}
        & *={}
        & *={}
        & *={} \ar @{-} [l] \ar @{~~} [ru]
        \\
        *={}
        & *={}
        & *={}
        & *={}
        & *={} \ar @{-} [u]
        & *={} \ar @{-} [l] \ar @{-} [u] \ar @{} [ul]|{1}
        \\
        *={} \ar @{~~} [rrruuu]
        \\
        *={}
        \\
        *={} \ar @{--} [uu]
        & *={} \ar @{--} [l] \ar @{~~} [rrruuu]
      }\endxy 
      ,
      \xy<0cm,1.25cm>\xymatrix @R .5cm @C .5cm{
        *={}
        & *={}
        & *={}
        & *={}
        & *={}
        & *={} \ar @{--} [lllll]
        \\
        *={}
        \\
        *={}
        \\
        *={}
        \\
        *={} \ar @{--} [uuuu]
        & *={} \ar @{--} [l] \ar @{~~} [rrrruuuu]
      }\endxy\
    \right).
    \]
    This means that in  $m_{(\tt,\uu,\vv)}$ the vertices in the bottom
    row labeled by $r$ and by $1$ are connected to different boxes
    inside $y_\nu$. Note that  $y_{(k)}=\sum_d\pm y_{(k-1)}d$ where $d$ runs
    through a set of right coset representatives of $\mathfrak{S}_{k-1}$ in
    $\mathfrak{S}_k$. 
    For every $1\leq i\leq k$ define
    \[
    (\tt'_i,\uu',\vv') = \left(\
      \xy<0cm,1.5cm>\xymatrix @R .5cm @C .5cm{
        *={}
        & *={}
        & *={}
        & *={}
        & *={}
        & *={}
        & *={} \ar @{--} [llllll]
        \\
        *={} \ar @{--} [u]
        & *={} \ar @{-} [l]
        & *={} \ar @{-} [l]
        & *={} \ar @{-} [l]
        & *={} \ar @{-} [l]
        & *={} \ar @{-} [l] \ar @{~~} [ru]
        \\
        *={} \ar @{-} [u]
        & *={} \ar @{-} [l] \ar @{-} [u] \ar @{} [ul]|{\dots}
        & *={} \ar @{-} [l] \ar @{-} [u] \ar @{} [ul]|{\hat a_i}
        & *={} \ar @{-} [l] \ar @{-} [u] \ar @{} [ul]|{\dots}
        & *={} \ar @{-} [l] \ar @{-} [u] \ar @{} [ul]|{\dots}
        \\
        *={}
        & *={}
        & *={}
        & *={} \ar @{-} [l] \ar @{~~} [ru]
        \\
        *={}
        & *={}
        & *={} \ar @{-} [u]
        & *={} \ar @{-} [l] \ar @{-} [u] \ar @{} [ul]|{a_i}
        \\
        *={} \ar @{--} [uuu]
        & *={} \ar @{-} [l] \ar @{~~} [ru]
        \\
        *={} \ar @{-} [u]
        & *={} \ar @{-} [l] \ar @{-} [u] \ar @{} [ul]|{r}
      }\endxy
      ,
      \xy<0cm,1.5cm>\xymatrix @R .5cm @C .5cm{
        *={}
        & *={}
        & *={}
        & *={}
        & *={}
        & *={}
        & *={} \ar @{--} [lll]
        \\
        *={}
        & *={}
        & *={}
        & *={}
        & *={}
        & *={} \ar @{-} [l] \ar @{~~} [ru]
        \\
        *={}
        & *={}
        & *={}
        & *={}
        & *={} \ar @{-} [u]
        & *={}
        \\
        *={} \ar @{~~} [rrruuu]
        \\
        *={}
        \\
        *={} \ar @{--} [uu]
        & *={} \ar @{-} [l] \ar @{~~} [rrruuu]
        \\
        *={} \ar @{-} [u]
        & *={} \ar @{-} [l] \ar @{-} [u] \ar @{} [ul]|{1}
      }\endxy 
      ,
      \xy<0cm,1.5cm>\xymatrix @R .5cm @C .5cm{
        *={}
        & *={}
        & *={}
        & *={}
        & *={}
        & *={} \ar @{--} [lllll]
        \\
        *={}
        \\
        *={}
        \\
        *={}
        \\
        *={} \ar @{--} [uuuu]
        & *={} \ar @{--} [l] \ar @{~~} [rrrruuuu]
      }\endxy\
    \right)
    \]
    where $\hat{a}_i$ means omitting $a_i$. 
    Then $m_{(\tt,\uu,\vv)}\cdot e = \sum_i \pm m_{(\tt_i',\uu',\vv')}$
    and similarly in $C^{(\lambda',\mu')}$ we have 
    $m_{\Res(\tt,\uu,\vv)}\cdot e = \sum_i \pm
    m_{\Res(\tt_i',\uu',\vv')}$. It might happen that
    $(\tt_i',\uu',\vv')$ is not standard. Using the same arguments as we did for the generators without horizontal edges,  $m_{(\tt,\uu,\vv)}\cdot e$ and
    $m_{\Res(\tt,\uu,\vv)}\cdot e$  can be respectively written  as
    a linear combination of basis elements and the claim follows by
    the same arguments as before. 
  \item
    Third case:  $r$ is an entry in row $l$ of $\tt$ and $1$ is an
    entry of $\vv$. 

    Let the entries in row $l$ of
    $\uu$ be $b_1,\dots,b_k$. 
    We have 
    \[
    (\tt,\uu,\vv) = \left(\
      \xy<0cm,1.25cm>\xymatrix @R .5cm @C .5cm{
        *={}
        & *={}
        & *={}
        & *={}
        & *={}
        & *={}
        & *={} \ar @{--} [llllll]
        \\
        *={}
        & *={}
        & *={}
        & *={}
        & *={}
        & *={} \ar @{-} [l] \ar @{~~} [ru]
        \\
        *={}
        & *={}
        & *={}
        & *={}
        & *={} \ar @{-} [u]
        & *={} \ar @{-} [l] \ar @{-} [u] \ar @{} [ul]|{r}
        \\
        *={}
        \\
        *={}
        \\
        *={} \ar @{--} [uuuuu]
        & *={} \ar @{--} [l] \ar @{~~} [rrruuu]
      }\endxy
      ,
      \xy<0cm,1.25cm>\xymatrix @R .5cm @C .5cm{
        *={}
        & *={}
        & *={}
        & *={}
        & *={}
        & *={}
        & *={} \ar @{--} [lll]
        \\
        *={}
        & *={}
        & *={} \ar @{~~} [ru]
        & *={}
        & *={}
        & *={} \ar @{-} [llll] \ar @{~~} [ru]
        \\
        *={}
        & *={} \ar @{-} [u]
        & *={} \ar @{-} [l] \ar @{-} [u] \ar @{} [ul]|{b_1}
        & *={} \ar @{-} [l] \ar @{-} [u] \ar @{} [ul]|{b_2}
        & *={} \ar @{-} [l] \ar @{-} [u] \ar @{} [ul]|{\dots}
        & *={} \ar @{-} [l] \ar @{-} [u] \ar @{} [ul]|{b_k}
        \\
        *={} \ar @{~~} [ru]
        \\
        *={}
        \\
        *={} \ar @{--} [uu]
        & *={} \ar @{--} [l] \ar @{~~} [rrruuu]
      }\endxy 
      ,
      \xy<0cm,1.25cm>\xymatrix @R .5cm @C .5cm{
        *={}
        & *={} \ar @{-} [l]
        & *={}
        & *={}
        & *={}
        & *={} \ar @{--} [llll]
        \\
        *={} \ar @{-} [u]
        & *={} \ar @{-} [l] \ar @{-} [u] \ar @{} [ul]|{1}
        \\
        *={}
        \\
        *={}
        \\
        *={} \ar @{--} [uuuu]
        & *={} \ar @{--} [l] \ar @{~~} [rrrruuuu]
      }\endxy\
    \right).
    \]
    For every $1\leq i\leq k$ define
    \[
    (\tt_i',\uu',\vv') = \left(\
      \xy<0cm,1.5cm>\xymatrix @R .5cm @C .5cm{
        *={}
        & *={}
        & *={}
        & *={}
        & *={}
        & *={}
        & *={} \ar @{--} [llllll]
        \\
        *={}
        & *={}
        & *={}
        & *={}
        & *={}
        & *={} \ar @{-} [l] \ar @{~~} [ru]
        \\
        *={}
        & *={}
        & *={}
        & *={}
        & *={} \ar @{-} [u]
        \\
        *={}
        \\
        *={}
        \\
        *={} \ar @{--} [uuuuu]
        & *={} \ar @{-} [l] \ar @{~~} [rrruuu]
        \\
        *={} \ar @{-} [u]
        & *={} \ar @{-} [l] \ar @{-} [u] \ar @{} [ul]|{r}
      }\endxy
      ,
      \xy<0cm,1.5cm>\xymatrix @R .5cm @C .5cm{
        *={}
        & *={}
        & *={}
        & *={}
        & *={}
        & *={}
        & *={} \ar @{--} [lll]
        \\
        *={}
        & *={}
        & *={} \ar @{~~} [ru]
        & *={}
        & *={}
        & *={} \ar @{-} [llll] \ar @{~~} [ru]
        \\
        *={}
        & *={} \ar @{-} [u]
        & *={} \ar @{-} [l] \ar @{-} [u] \ar @{} [ul]|{\dots}
        & *={} \ar @{-} [l] \ar @{-} [u] \ar @{} [ul]|{\hat b_i}
        & *={} \ar @{-} [l] \ar @{-} [u] \ar @{} [ul]|{\dots}
        \\
        *={} \ar @{~~} [ru]
        \\
        *={}
        \\
        *={} \ar @{--} [uu]
        & *={} \ar @{-} [l] \ar @{~~} [rrruuu]
        \\
        *={} \ar @{-} [u]
        & *={} \ar @{-} [l] \ar @{-} [u] \ar @{} [ul]|{1}
      }\endxy 
      ,
      \xy<0cm,1.5cm>\xymatrix @R .5cm @C .5cm{
        *={}
        & *={} \ar @{-} [l]
        & *={}
        & *={}
        & *={}
        & *={} \ar @{--} [llll]
        \\
        *={} \ar @{-} [u]
        & *={} \ar @{-} [l] \ar @{-} [u] \ar @{} [ul]|{b_i}
        \\
        *={}
        \\
        *={}
        \\
        *={} \ar @{--} [uuuu]
        & *={} \ar @{--} [l] \ar @{~~} [rrrruuuu]
      }\endxy\
    \right).
    \]
		Note that $\shape(\tt_i',\uu',\vv')=(\lambda,\mu)$. 
	We have $m_{(\tt,\uu,\vv)}\cdot e = \sum_i \pm
    m_{(\tt_i',\uu',\vv')}$. Note the special subcase 
    if row $l$ in $\uu$ is empty. In this case $e$ acts as zero
    on $m_{(\tt,\uu,\vv)}$. Again the claim follows. 
\end{itemize}

\end{proof}


%% file: annihilatorbasis.tex
\section{A basis adapted to the annihilator}\label{section:annihilatorbasis}

In this section, we define another basis of the walled Brauer
algebra. Again, this basis is indexed by tuples of standard
triples. However, if $x$ is specialized to $n$, then a subset of this
basis is a basis for the annihilator on the mixed tensor space. 

From Corollary~\ref{cor:rank} we see that $\max(Y)$ for a path $Y$
plays an important role in a combinatorial description for a
basis of the annihilator.  
\begin{defn}
  Let  $(\tt,\uu,\vv)$ be a standard triple of shape $(\lambda,\mu)$.
  For $i=0,\ldots, s$ let $u_i$ be the number of entries in the first
  row of $\uu$ which are $> i$ and let $v_i$ be the number
  of entries in the first
  row of $\vv$ which are $\leq i$. Let $m_i=\lambda_1+u_i+v_i$.  
 Then let $\max (\tt,\uu,\vv)$ be the maximum of
 $\{m_0,\ldots,m_{s}\}$.

\end{defn}

If $Y$ is a path and  $(\tt,\uu,\vv)$ is the corresponding standard
triple, then $\max(Y)= \max (\tt,\uu,\vv)$. 

\begin{lemma}\label{lemma:construction_of_basis}
\begin{enumerate}
\item
  Let $R=\mathbb{Z}[x]$. 
  Let  $(\tt,\uu,\vv)$
  be  a standard triple of shape
  $(\lambda,\mu)$ with $\lambda\vdash (r-k)$ and
  $\mu\vdash (s-k)$. 
  Then  for
  each standard triple $(\tilde\tt,\tilde\uu,\tilde\vv)$  of shape
  $(\lambda,\mu)$ there exists a coefficient 
  $r_{(\tilde\tt,\tilde\uu,\tilde\vv)}\in R$
  such that
	\begin{itemize}
	\item
  $r_{(\tilde\tt,\tilde\uu,\tilde\vv)}\neq 0$ only if
  $\ss\rhd^{\op} \tilde\ss$ (defined as above) or $(\tilde\tt,\tilde\uu,\tilde\vv)=(\tt,\uu,\vv)$, 
	\item
  $r_{(\tt,\uu,\vv)}=1$,  
	\item
	if  $(\tt',\uu',\vv')$ is a standard triple of shape
  $(\lambda,\mu)$ then there exists  an element $b\in
  B^{\rhd(\lambda,\mu)}$ such that
  \begin{equation}\label{equation:annihilator}
    \sum  r_{(\tilde\tt,\tilde\uu,\tilde\vv)}
    m_{(\tt',\uu',\vv'),(\tilde\tt,\tilde\uu,\tilde\vv)}+b
  \end{equation}
  can be written as a linear combination of elements each of which
  involves $y_{(\max(\tt,\uu,\vv))}$. 
	\end{itemize}
  Note that the coefficients depend on $(\tt,\uu,\vv)$ but not on
  $(\tt',\uu',\vv')$. We write
  $r_{(\tilde\tt,\tilde\uu,\tilde\vv)}^{(\tt,\uu,\vv)}$ instead of 
  $r_{(\tilde\tt,\tilde\uu,\tilde\vv)}$
   to emphasize
  that the coefficients depend on $(\tt,\uu,\vv)$.

\item 
  If $R$ is arbitrary, then the element in
  \eqref{equation:annihilator} is defined by specializing
  coefficients. 
  In particular, if
  $x=n=\mathrm{rank}(V)<\max(\tt,\uu,\vv)$
  then this element is an element  of the annihilator of
  the walled Brauer algebra on mixed tensor space. 
\end{enumerate}
\end{lemma}

\begin{proof}
  The second part follows by
  Remark~\ref{remark:action_of_the_walled_Brauer_algebra}, so let
  $R=\mathbb{Z}[x]$. Since the upper part of the diagrams is not involved in the calculations, we show a similar  result for
  $m_{(\tt,\uu,\vv)}$ instead of
  $m_{(\tt',\uu',\vv'),(\tt,\uu,\vv)}$, then the lemma 
  follows by premultiplying $b_{(\tt',\uu',\vv')}^*$ . 
 We write $\max$   instead of $\max (\tt,\uu,\vv)$ and choose $i_0$
 such that  $m_{i_0}= \max$. 

  Let $u=u_{i_0}$ and $v=v_{i_0}$. Then
  we have $\max ={\lambda_1+u}+v$. Note that ${\lambda_1+u}\leq \nu_1$
  and $v\leq \mu_1$.  
  By choosing coset representatives we get
  $z_1\in R\mathfrak{S}_r$ and $z_2\in R\mathfrak{S}_{s-k}$ such that
  
  \[
  \wtangle{annihilator.3}{2.5cm}= \wtangle{annihilator.4}{2.5cm}
  \quad\text{ and }\quad
  \wtangle{annihilator.6}{2.5cm} = \wtangle{annihilator.5}{2.5cm} 
  \]

  Thus we have
  \[
  m_{(\tt,\uu,\vv)}=
  \htangle{annihilator.1}{4cm} =\htangle{annihilator.2}{4cm}
  \]
  There are ${\lambda_1+u}$ edges ending in $y_{({\lambda_1+u})}$,
  $\lambda_1$ of these  edges
  start at the first $\lambda_1$ upper vertices on the left, the
  other edges start at the vertices on the   right in the bottom row
  namely at the vertices numbered by the first $u$ entries in the 
  first row of $\uu$.

  Note that
  \[
  \wtangle{annihilator.8}{2cm} \;=\;\sum_d\sgn{d}\;\wtangle{annihilator.7}{2cm}
  \]
  where the sum runs over a set of right coset representatives of
  $\mathfrak{S}_{\lambda_1+u}\times \mathfrak{S}_v$ in $\mathfrak{S}_{\max}$ containing the identity $\mathrm{id}$. Let
  $m_d$ be the element obtained from $
  m_{(\tt,\uu,\vv)}$ by adding a box filled with the permutation $d$
  below the boxes $y_{({\lambda_1+u})}$ and  $y_{(v)}$ in the right
  diagram.
  Then $\sum_d\sgn{d} m_d$ contains
  $y_{(\max)}$. Certainly if $d=\mathrm{id}$, then
  $m_d=m_{(\tt,\uu,\vv)}$.

  Suppose first, that $d$ connects the first $\lambda_1$ edges coming from
  $d^{\op}(\oo)^*$ to the box $y_{{\lambda_1+u}}$.
  Then $m_d$ is up to a sign equal to
  $m_{(\tilde\tt,\tilde\uu,\tilde\vv)}$ where $\tilde\tt=\tt$ and
  $\tilde\uu$ and $\tilde \vv$ 
  are obtained from $\uu$ and $\vv$ by a permutation on the set of  the first $u$
  entries in the first row of $\uu$ together with the
  first $v$ entries in the 
  first row of $\vv$. Since these entries in $\uu$ are
  $> i_0$ and those in $\vv$ are $\leq i_0$
  we have $\ss\rhd^{\op}\tilde\ss$ for $d\neq\mathrm{id}$. If
  $(\tilde\tt,\tilde\uu,\tilde\vv)$ is not yet standard then one can
  rewrite it as a linear combination of standard elements still
  satisfying  $\ss\rhd^{\op}\tilde\ss$
  (c.~f.~Remark~\ref{remark:straightening}) and strictly dominating terms.  

  Suppose now that  $d$ connects one of the first $\lambda_1$ edges coming from
  $d^{\op}(\oo)^*$ to the box $y_{(v)}$. Then this edge is an additional
  horizontal edge and we have $b_{(\tt',\uu',\vv')}^*m_d\in B^{\rhd(\lambda,\mu)}$.  

\end{proof}

\begin{defn}
  Lemma~\ref{lemma:construction_of_basis} shows that for each pair
  $(\tt,\uu,\vv)$ and $(\tt',\uu',\vv')$
  of standard triples there is an element $b\in   B^{\rhd(\lambda,\mu)}$, such
  that
  \begin{equation}\label{equation:construction_of_basis}
    \sum_{(\tilde\tt,\tilde\uu,\tilde\vv), (\tilde{\tilde\tt},\tilde{\tilde\uu},\tilde{\tilde\vv})} 
    r_{(\tilde{\tilde\tt},\tilde{\tilde\uu},\tilde{\tilde\vv})}^{(\tt',\uu',\vv')}
    r_{(\tilde\tt,\tilde\uu,\tilde\vv)}^{(\tt,\uu,\vv)}
    m_{(\tilde{\tilde\tt},\tilde{\tilde\uu},\tilde{\tilde\vv}),(\tilde\tt,\tilde\uu,\tilde\vv)}+b
  \end{equation}
  involves $y_{(\max(\tt,\uu,\vv))}$. Similarly, $b$ can be chosen
  such that the element in \eqref{equation:construction_of_basis}
  involves  $y_{(\max(\tt',\uu',\vv'))}$.

  For each pair of standard triples let the element
  $c_{(\tt',\uu',\vv'),(\tt,\uu,\vv)}$ 
  be one of these two elements such that
  $c_{(\tt',\uu',\vv'),(\tt,\uu,\vv)}$ involves $y_{(\max)}$ where
  $\max$ is the maximum of $ \max(\tt,\uu,\vv)$ and
  $\max(\tt',\uu',\vv')$. Note that this definition then makes sense for
  all commutative rings $R$ with one.  
\end{defn}

\begin{thm}\label{main_thm:weakly_cellular_basis}
  \begin{enumerate}
  \item Let $R$ be a commutative ring with one. Then 
	\[\{c_{(\tt',\uu',\vv'),(\tt,\uu,\vv)}\mid 
    (\tt',\uu',\vv'),(\tt,\uu,\vv)\in M(\lambda,\mu),(\lambda,\mu)\in
    \Lambda(r,s)\}\] is a weakly cellular
    basis of the $R$-algebra $B_{r,s}(x)$ with anti-involution $*$.
  \item If $R$ is such that $x=n$, then the annihilator of
    $B_{r,s}(n)$ on mixed tensor space is free with basis 
    \[
    \left\{c_{(\tt',\uu',\vv'),(\tt,\uu,\vv)}\mid 
    \begin{array}{l}
      (\tt',\uu',\vv'),(\tt,\uu,\vv)\in M(\lambda,\mu),(\lambda,\mu)\in
      \Lambda(r,s),\\
      \max(\tt,\uu,\vv)\text{ or }\max(\tt',\uu',\vv')>n
    \end{array} \right\}.
    \]
  \item For $x=n$ let $\mathrm{ann}= \mathrm{ann}_{B_{r,s}(n)}(V^{\otimes r}\otimes
    {V^*}^{\otimes s})$ be the annihilator of the
    walled Brauer algebra on mixed tensor space and let 
    $\overline{c}_{(\tt',\uu',\vv'),(\tt,\uu,\vv)}\in
    B_{r,s}(n)/\mathrm{ann}$
    be the coset of
    $c_{(\tt',\uu',\vv'),(\tt,\uu,\vv)}$ modulo the annihilator. Then
    \[
    \left\{ \overline{c}_{(\tt',\uu',\vv'),(\tt,\uu,\vv)}\mid 
    \begin{array}{l}
      (\tt',\uu',\vv'),(\tt,\uu,\vv)\in M(\lambda,\mu),(\lambda,\mu)\in
      \Lambda(r,s),\\
      \max(\tt,\uu,\vv),\max(\tt',\uu',\vv')\leq n
    \end{array} \right\}
    \] is a weakly cellular
    basis of the $R$-algebra
    $B_{r,s}(x)/\mathrm{ann}
    \cong\mathrm{End}_{U}(V^{\otimes r}\otimes {V^*}^{\otimes s})$. 
   \end{enumerate}
 \end{thm}
\begin{proof}
  The base change matrix between the two bases
  $\{c_{(\tt',\uu',\vv'),(\tt,\uu,\vv)}\}$ and  $\{m_{(\tt',\uu',\vv'),(\tt,\uu,\vv)}\}$
  is uni-triangular since
  $r_{(\tilde\tt,\tilde\uu,\tilde\vv)}^{(\tt,\uu,\vv)}\neq 0$ implies
  that $\ss\rhd^{\op} \tilde\ss$ or $(\tilde\tt,\tilde\uu,\tilde\vv)=(\tt,\uu,\vv)$  and
  $r_{(\tt,\uu,\vv)}^{(\tt,\uu,\vv)}=1$. Thus
  $\{c_{(\tt',\uu',\vv'),(\tt,\uu,\vv)}\}$ is in fact a basis.  Since 
  $c_{(\tt',\uu',\vv'),(\tt,\uu,\vv)}^*-c_{(\tt,\uu,\vv),(\tt',\uu',\vv')}\in
  B^{\rhd(\lambda,\mu)}$ by
  Equation~\eqref{equation:construction_of_basis}, Condition~C1' is
  satisfied. Condition~C2 can be easily verified. 

  The second part then follows
  since the basis elements are in fact elements of the annihilator,
  the basis has the right cardinality and the annihilator is $R$-free
  with $R$-free complement.
  The third part is a direct consequence.

\end{proof}
\begin{exmp}
	Recall the example in Figure  from the introduction for $r=s=2$. The elements of $\Lambda(2,2)$ are $((2),(2))$, $((1^2),(2))$, $((2),(1^2))$, $((1^2),(1^2))$ (for $k=0$), $((1),(1))$ (for $k=1$) and $(\emptyset,\emptyset)$ for ($k=2$), from top to bottom. The reader may list all triples of this shape and the corresponding maximum to obtain the diagrams in the figure.
\end{exmp}

As a direct consequence of
Theorem~\ref{main_thm:weakly_cellular_basis} we get
 \begin{thm}\label{theorem:cellular_basis_endom}
   Let $\overline{m}_{(\tt',\uu',\vv'),(\tt,\uu,\vv)}$
   be the coset of $m_{(\tt',\uu',\vv'),(\tt,\uu,\vv)}$ modulo $\mathrm{ann}$. Then 
   \[
   \left\{ \overline{m}_{(\tt',\uu',\vv'),(\tt,\uu,\vv)}\mid 
     \begin{array}{l}
       (\tt',\uu',\vv'),(\tt,\uu,\vv)\in M(\lambda,\mu),(\lambda,\mu)\in
       \Lambda(r,s),\\
       \max(\tt,\uu,\vv),\max(\tt',\uu',\vv')\leq n
     \end{array} \right\}
   \] is a  cellular
   basis of $B_{r,s}(x)/\mathrm{ann}
   \cong\mathrm{End}_{U}(V^{\otimes r}\otimes {V^*}^{\otimes s})$ with
   anti-involution induced by $*$.   
 \end{thm}
\begin{proof}
  Let
  \[
  \tilde{c}_{(\tt',\uu',\vv'),(\tt,\uu,\vv)}
  =\begin{cases}
    m_{(\tt',\uu',\vv'),(\tt,\uu,\vv)}&\text{ if }
    \max(\tt,\uu,\vv),\max(\tt',\uu',\vv')\leq n\\
    c_{(\tt',\uu',\vv'),(\tt,\uu,\vv)}&\text{ otherwise}.
  \end{cases}
  \]
  Again, $\{\tilde{c}_{(\tt',\uu',\vv'),(\tt,\uu,\vv)}\}$ is a basis
  since the base change matrix is uni-triangular. A subset of this
  basis is a basis of the annihilator and thus this set is actually a basis of $B_{r,s}(x)/\mathrm{ann}$. 
  Condition~C1 is obvious, Condition~C2 follows using
  Equation~\eqref{equation:annihilator} inductively. 
 
\end{proof}

\begin{remark}
  Note that the partially ordered set corresponding to these cellular bases of  $B_{r,s}(x)/\mathrm{ann}$  is
  $\Lambda(r,s)$. The set corresponding to an element $(\lambda,\mu)$ of $\Lambda(r,s)$
  is a subset of $M(\lambda,\mu)$, namely the set of standard triples
  $(\tt,\uu,\vv)$ with
  $\max(\tt,\uu,\vv)\leq n$. We denote this set by
  $M_0(\lambda,\mu)$.  Note that this set might be empty. 
  To be precise, $M_0(\lambda,\mu)=\emptyset$ if and only
  if $\lambda_1+\mu_1>n$.   
  
  Let $\Lambda_0(r,s)$ be the set of pairs
  $(\lambda,\mu)\in\Lambda(r,s)$ such that $\lambda_1+\mu_1\leq
  n$. Then the cellular basis of $\mathrm{End}_{U}(V^{\otimes
    r}\otimes {V^*}^{\otimes s})$ is indexed by pairs of elements of
  $M_0(\lambda,\mu)$ with $(\lambda,\mu)\in\Lambda_0(r,s)$ and we can take $\Lambda_0(r,s)$ as corresponding partially ordered set. 
\end{remark}


%% file: filtration.tex
\section{A basis and a filtration for ordinary tensor space}\label{section:ordinary_tensor_space}
Before we turn back to the walled Brauer algebra and mixed tensor space, we state some results on ordinary tensor space adjusted for our purposes. Let $m$ be a natural number, then $V^{\otimes m}$ is called the \emph{ordinary tensor space}.

\begin{defn}
	Let $\lambda$ be a partition of $m$ and $\mu$ a composition of $m$ into $n$ parts. A $\lambda$-tableau of type $\mu$ is a filling of the boxes of $[\lambda]$ with not necessarily distinct numbers such that the number of entries equal to $i$ is $\mu_i$. 	

A $\lambda$-tableau of type $\mu$ is called \emph{semi-standard} if it is row-standard and the entries
 weakly increase along the columns 
 (see Figure~\ref{figure:semistandardtableau}). Let $\Tab$ be the set of all tableaux of some type, let $\Tab(\lambda)$ be the set of $\lambda$-tableaux of some type,  and let $\Semistd(\lambda)$ denote 
 the set of semi-standard $\lambda$-tableaux of some type.
\end{defn}

\begin{figure}[h!]
  \centering
  \epsfbox{figures/youngtableaux.18}
   \caption{A semi-standard $(3,3,2,2)$-tableau of type $(2,3,1,3,1)$}\label{figure:semistandardtableau}
\end{figure}

\begin{defn}
	Let $\lambda\vdash m$ be a partition of $m$, $T\in\Tab(\lambda)$
	 and $\tt\in\tab(\lambda)$ be row-standard tableaux. Let $v_T=v_{\bf i}\in V^{\otimes m}$ where $\bf i$ is the multi-index obtained from $T$ by reading the entries row by row and let  $v_{T\tt}:=v_T y_\lambda d(\tt)\in V^{\otimes m} $.    
\end{defn}
Note that $v_{T\tt}$ is a linear combination of those $v_{\bf i}$ with coefficients $1$ or $-1$ where $\bf i$ is a multi-index such that for each row in $[\lambda]$  the numbers in $\tt$ in this row indicate the positions of the numbers in $T$ in this row. 
\begin{defn}
If $T,S\in\Tab$ are row-standard, we say that $T$ dominates $S$ ($T\trianglerighteq S$) if the shape of $T\downarrow i$ dominates the shape of $S\downarrow i$ for all $i=1,\ldots,n$. Here $T\downarrow i$ is the tableau obtained from $T$ by removing all boxes with entries greater than $i$. 
\end{defn}

The following result can be shown by similar methods as the results in \cite{murphy} where a  basis for  the Hecke 
algebra of the symmetric group with similar properties is given.  
\begin{thm}\label{theorem:Straigthening}
\begin{enumerate}
\item\label{item:basis_of_tensor_space} The set
\[\{v_{T\tt}\mid\lambda\vdash m, T\in \Semistd(\lambda),\tt\in\Std(\lambda)\}\]
is a basis of $V^{\otimes m}$.  

\item\label{item:straightening_T} Let $\lambda\vdash m$, $T\in\Tab(\lambda)$
 and $\tt\in\tab(\lambda)$  such that $T$ and $\tt$ are row-standard but 
 $T$ is not semi-standard.
 Then $v_{T\tt}$ is a linear combination of $v_{S\ss}$ where either $\ss=\tt$ and $S\triangleright T$ or $S$ and $\ss$ are tableaux of shape $\lambda'$ where $\lambda'\triangleright \lambda$.

\item\label{item:straightening_t}Let $\lambda\vdash m$, $T\in\Tab(\lambda)$
 and $\tt\in\tab(\lambda)$ with $T$ and $\tt$ row-standard, but now suppose that
 $\tt$ is not standard. Then $v_{T\tt}$ is a linear combination of $v_{S\ss}$ where either $\ss\triangleright\tt$ and $S=T$ or $S$ and $\ss$ are tableaux of shape $\lambda'$ where $\lambda'\triangleright \lambda$.
\end{enumerate}

\end{thm}
\begin{proof}
	That the set in \ref{item:basis_of_tensor_space}.~is a generating system can be seen similarly as the results in \cite{murphy}: obviously, the set of all $v_{T\tt}$ with $T$ and $\tt$ row-standard $\lambda $-tableaux for partitions $\lambda$ generate the tensor space (all vectors $v_{\bf i}$ can be written as $v_{T\tt}$ with $\lambda=(1^m)$).
	
If $T$ is not semi-standard or $\tt$ is not standard then using Garnir relations
$v_{T\tt}$ can be rewritten as a linear combination similar as in 
\ref{item:straightening_T}.~or \ref{item:straightening_t}.~except that $\triangleright$ should be replaced by $>$ where $>$ is the lexicographical order. 	
 
To show that this set is linearly independent and that the other two parts of the theorem hold, it suffices to find a bilinear form $\langle.,.\rangle:V^{\otimes m}\times V^{\otimes m}\to R$ and vectors $w_{T,\tt}\in V^{\otimes m}$ for $T\in\Semistd(\lambda)$,  $\tt\in\Std(\lambda)$, $\lambda\vdash m$,
 such that the following holds:
  If $\lambda,\lambda'\vdash m$, $T\in \Semistd(\lambda)$, 
 $\tt\in\Std(\lambda)$, $S\in \Tab(\lambda')$
  row-standard and $\ss\in\tab(\lambda')$ row-standard, then we have 
\begin{itemize}
\item $	\langle w_{T\tt},v_{T\tt}\rangle$ is invertible. 
\item $	\langle w_{T\tt},v_{S\ss}\rangle\neq 0\Rightarrow T\trianglerighteq S$ and $\tt\trianglerighteq\ss$.    
\end{itemize} 
Let $\langle.,.\rangle$ be the bilinear form on $ V^{\otimes m}$ given by $\langle v_{\bf i},v_{\bf j}\rangle=\delta_{\bf i,\bf j} $ and let $w_{T\tt}$
 be the basis elements defined similarly to \cite{mathas}:
 
 Let $w_{T\tt}$ be the sum of all $v_{\bf i}$ where $\bf i$ is a multi-index such that for each column in $[\lambda]$  the numbers in $\tt$ in this column indicate the positions of the numbers in $T$ in this column. To compute $	\langle w_{T\tt},v_{S\ss}\rangle$ one has to find all multi-indices $\mathbf i$ such that for each row (column) in $[\lambda]$  the numbers in $\ss$ ($\tt$) in this row (column) indicate the positions of the numbers in $S$ ($T$) in this row (column).  Each such multi-index contributes $1$ or $-1$ to  $	\langle w_{T\tt},v_{S\ss}\rangle$. 
Let $\I$ be the set of multi-indices satisfying these conditions.

Suppose first that $S=T$ and $\ss=\tt$. Let $k$ be the largest entry of $T$ and let $j$ be maximal such that $k$ appears in column $j$ in $T$. Suppose the entries equal to $k$ in this column are in row $i+1,i+2,\ldots,i+l$.   The position of the $l$ entries equal to $k$ in $\bf i\in\I$ are numbers occuring in column $j$ of $\tt$. Since 
 no entry in the first $i$ rows of $T$ is equal to $k$, the numbers in the first $i$ rows of $\tt$ indicate positions with entries not equal to $k$. This uniquely determines the position of these $l$ entries in $\bf i$. 
Repeating this process shows that there is a unique multi-index ${\bf i}\in \I$ which shows that $\langle w_{T\tt},v_{T\tt}\rangle=\pm 1$ is invertible.

Suppose now  that $\tt\ntrianglerighteq \ss$ 
and there exists a multi-index ${\bf i }\in\I$. By \cite{james,murphy} there are two entries, say  $k$ and $l$ in some 
column $i$ of $\tt$ which are in the same row $j$ of $\ss$.
 Since $\bf i$ satisfies the conditions above, $i_k$ and $i_l$ are in the $i$-th column of $T$ and in the $j$-th row of $S$. In particular, $i_k\neq i_l$ and 
  ${\bf i }(kl)\in \I$. Since ${\bf i}(kl)$ contributes $-1$ times the summand that is contributed by ${\bf i}$ to $\langle w_{T\tt},v_{S\ss}\rangle$, we have $\langle w_{T\tt},v_{S\ss}\rangle=0$. 
  
	  
Suppose now that $T\ntrianglerighteq S$ and there is a multi-index ${\bf i}\in\I$. Since $T\ntrianglerighteq S$ there is some  $i$ such that $\shape (T\downarrow i)\ntrianglerighteq \shape (S\downarrow i)$. Delete all entries in $\ss$ and in $\tt$ that indicate positions in $\bf i$ of entries greater than $i$.  
Then each row of the tableau obtained from $\ss$ has as many boxes as the corresponding row in $S$ has, a similar statement holds for columns, $\tt$ and $T$.
These new tableaux might have holes, but after pushing together entries in each row (column) of the tableau obtained from  $\ss$ ($\tt$), these tableaux have the same shape as $S\downarrow i$ ($T\downarrow i$).
 Again by the results in \cite{james,murphy}	   there are two entries in some 
column $i$ of $\tt$ which are in the same row $j$ of $\ss$. Now $\langle w_{T\tt},v_{S\ss}\rangle=0$ follows as before.

\end{proof}

\section{A filtration of the mixed tensor space}\label{section:filtration_mixed_tensor_space}
Fix a natural number $n$. In this section, we define a basis  of the mixed tensor space which leads to a filtration of the mixed tensor space  with $U$-$B_{r,s}(n)$-bimodules. Each layer of this filtration is a tensor product of a $U$-left module and a $B_{r,s}(n)$-right module.

Recall the cellular basis of
$\Endo:=\mathrm{End}_{U}(V^{\otimes r}\otimes {V^*}^{\otimes s})$
 \[
   \left\{ \overline{m}_{(\tt',\uu',\vv'),(\tt,\uu,\vv)}\mid 
    	(\tt',\uu',\vv'),(\tt,\uu,\vv)\in M_0(\lambda,\mu),(\lambda,\mu)\in
       \Lambda_0(r,s) \right\}
   \]
from Theorem~\ref{theorem:cellular_basis_endom}.
\begin{defn} 
For $(\lambda,\mu)\in\Lambda_0(r,s)$ let  
$\overline{C}^{(\lambda,\mu)}$ be the \emph{cell module} of
$\Endo$
with $R$-basis $\{c_{(\tt,\uu,\vv)}\mid (\tt,\uu,\vv)\in M_0(\lambda,\mu)\}$. By inflation, it
is also a $B_{r,s}(n)$-right module with action given by
       \[
       c_{(\tt,\uu,\vv)}a =	\sum_{(\tt',\uu',\vv')}\lambda_{(\tt',\uu',\vv')}c_{(\tt',\uu',\vv')}     
       \]
for $a\in B_{r,s}(n)$ and $\lambda_{(\tt',\uu',\vv')}\in R$ determined by
	\[\overline{m}_{(\tt'',\uu'',\vv''),(\tt,\uu,\vv)}\overline{a}\equiv\sum_{(\tt',\uu',\vv')}\lambda_{(\tt',\uu',\vv')}\overline{m}_{(\tt'',\uu'',\vv''),(\tt',\uu',\vv')}\mod \Endo(\rhd (\lambda,\mu)),\]
where $\overline{a}$ is the coset of $a$ modulo $\mathrm{ann}$ and thus an element of  
$\Endo$.
\end{defn}

\begin{defn}\label{defn:rationaltableaux}
  Let $0\leq k\leq \mathrm{min}(r,s)$ and let  
  $\lambda\in \Lambda(r-k)$ and $\mu\in
  \Lambda(s-k)$ be partitions with $\lambda_1+\mu_1\leq n$.
  A \emph{rational $(\lambda,\mu)$-tableau} is a pair
  $(\aa,\bb) \in\mathcal{T}(\lambda)\times\mathcal{T}(\mu)$. 
  
  Let $\mathrm{first}_i(\aa,\bb)$ be the number 
  of entries of the first row of $\aa$ which are $\leq i$
  plus the number 
  of entries of the first row of $\bb$ which are $\leq i$. 
  A rational tableau is called \emph{standard} if 
  $\aa\in\mathcal{T}_0(\lambda)$, $\bb\in\mathcal{T}_0(\mu)$
  and the following condition  holds:
  \begin{equation}
    \mathrm{first}_i(\aa,\bb)\leq i 
    \text{ for all }i=1,\ldots,n
    \label{equ:condition}
  \end{equation}
  We denote the set of standard rational $(\lambda,\mu)$-tableaux by
  $\Rat(\lambda,\mu)$. 
\end{defn}

\begin{remark}\label{remark:bijection_rational}
  Let $(\lambda,\mu)$, $\lambda=(\lambda_1,\ldots,\lambda_{l_\lambda})$,
  $\mu=(\mu_1,\ldots,\mu_{l_\mu})$ be as in
  Definition~\ref{defn:rationaltableaux} and 
  let $l\geq l_\mu$ be a natural number. 
  Let 
  \[
  \tau=(n^{l-l_\mu},n-\mu_{l_\mu},\ldots,
  n-\mu_1,\lambda_1,\ldots,\lambda_{l_\lambda}). 
  \]
  Then $\mathcal{T}_0(\tau)$ 
   is in bijection with the set of standard rational
  $(\lambda,\mu)$-tableaux (see \cite{dipperdotystoll2}). By \cite{stembridge}, the
  cardinality of this set is equal to the dimension of
  $V_{\lambda,\mu}$ in Proposition~\ref{prop:stembridge}. 
\end{remark}

\begin{defn}
  Let $(\aa,\bb)\in \Rat(\lambda,\mu)$ and $(\tt,\uu,\vv)\in
  M(\lambda,\mu)$. 
  Let $a_1,a_2$, \ldots, $a_{r-k}$ and $b_1,\ldots,b_{s-k}$ be the entries
  of $\aa$ and $\bb$ respectively read row by row. Then
  $v_{\aa,\bb}:=v_{a_{r-k}}\otimes \cdots \otimes v_{a_2}\otimes v_{a_1}\otimes
  v_{b_1}^*
  \otimes v_{b_2}^*\otimes\cdots\otimes  v_{b_s-k}^*$ is an element of 
  $V^{\otimes r-k}\otimes {V^*}^{\otimes s-k}$. Set
  \[
  v_{(\aa,\bb),(\tt,\uu,\vv)}:=v_{\aa,\bb}m_{(\tt,\uu,\vv)}
  \in V^{\otimes r}\otimes {V^*}^{\otimes s}
  \]

\end{defn}

\begin{thm}
  The set
  \[
  \{
   v_{(\aa,\bb),(\tt,\uu,\vv)}\mid 
   (\aa,\bb)\in \Rat(\lambda,\mu),
   (\tt,\uu,\vv)\in M_0(\lambda,\mu), (\lambda,\mu)\in \Lambda_0(r,s)
  \}
  \]
  is an $R$-basis of mixed tensor space $V^{\otimes r}\otimes
  {V^*}^{\otimes s}$. 
\end{thm}

\begin{proof}
  Since the rank of  $V_{\lambda,\mu}$  in
  Proposition~\ref{prop:stembridge} for
  $(\lambda,\mu)\in\Lambda_0(r,s)$
  is equal to the cardinality of $\Rat(\lambda,\mu)$   and
  the multiplicity of $V_{\lambda,\mu}$ in $V^{\otimes r}\otimes
  {V^*}^{\otimes s}$ is the cardinality of $M_0(\lambda,\mu)$, the
  cardinality of the set $\{ v_{(\aa,\bb),(\tt,\uu,\vv)}\}$ is equal
  to the rank of the mixed tensor space $V^{\otimes r}\otimes
  {V^*}^{\otimes s}$.  Thus it suffices to show that the set $\{
  v_{(\aa,\bb),(\tt,\uu,\vv)}\}$ generates $V^{\otimes r}\otimes
  {V^*}^{\otimes s}$.

  The elements  $v_{(\aa,\bb),(\tt,\uu,\vv)}$ can be defined even if
  $(\aa,\bb)$ and $(\tt,\uu,\vv)$ are not standard. It is easy to see
  that the $v_{(\aa,\bb),(\tt,\uu,\vv)}$ with all tableaux
  row-standard/row-anti-standard generate 
  $V^{\otimes r}\otimes {V^*}^{\otimes s}$. So we have to show that if
  $(\aa,\bb)$ or $(\tt,\uu,\vv)$ is not standard, then
  $v_{(\aa,\bb),(\tt,\uu,\vv)}$
  can be
  written as a linear combination of such elements involving tableaux
  which are dominating in some sense. 

  Let $(\aa,\bb),(\aa',\bb')\in \Rat(\lambda,\mu)$ and $(\tt,\uu,\vv),
  (\tt',\uu',\vv')\in M_0(\lambda,\mu)$. We define  
  $((\aa,\bb),(\tt,\uu,\vv))\trianglerighteq ((\aa',\bb'),(\tt',\uu',\vv'))$
  if and only if $(\lambda,\mu)\rhd(\lambda',\mu')$
  or $(\lambda,\mu)=(\lambda',\mu')$,  $\aa\trianglerighteq\aa'$,
  $\bb\trianglerighteq\bb'$ and 
  $(\tt,\uu,\vv)\trianglerighteq (\tt',\uu',\vv')$. 

	Now if $(\tt,\uu,\vv)$ is not standard, by the previous results 
$v_{(\aa,\bb),(\tt,\uu,\vv)}$ can be written as a linear combination of dominating elements with respect to the order just defined. Suppose that $\aa$ is not semi-standard. 
Let $v_{\aa,\tt^\lambda}$ be the basis element as in 
Section~\ref{section:ordinary_tensor_space}. Then  $v_{(\aa,\bb),(\tt,\uu,\vv)}$  involves a modified version $v_{\aa,\tt^\lambda}^{\mathrm{refl}}$ of this basis  element with reversed ordering of the tensor product, i.~e.~it can be written as $(v_{\aa,\tt^\lambda}^{\mathrm{refl}}\otimes \ldots)\cdot \ldots$. 
Using Theorem~\ref{theorem:Straigthening}, $v_{(\aa,\bb),(\tt,\uu,\vv)}$ can be written as a linear combination of dominating elements in the above sense. The same works if $\bb$ is not semi-standard.

 Finally, suppose that Condition~\eqref{equ:condition} does not hold and $i$ is the minimal entry violating this condition. 
We first assume, that $\aa$ and $\bb$ are tableaux with one row and $i$ is the greatest entry in $\aa$ and $\bb$. Let $I=\{i_1,i_2,\ldots,i_l\}$ where $i_1<i_2<\ldots <i_l=i$ are the entries appearing both in $\aa$ and $\bb$. Let $D=I\cup\{i+1,i+2,\ldots,n\}$. Note that each element $v_{(\aa,\bb),(\tt,\uu,\vv)}$ can be written as $v_{\aa,\bb}m_{(r-k),(s-k)}b$ for some linear combination of generalized diagrams $b$. Let $M$ be the set of all rational  $((r-k),(s-k))$-tableaux $(\aa',\bb')$ such that $\aa'$ and $\bb'$ are row-standard and $\aa'$ and $\bb'$ are obtained from $\aa$ and $\bb$ by replacing the entries in $I$ by entries in $D$. In particular $(\aa,\bb)\in M$ and all other elements of $M$ satisfy Condition~\eqref{equ:condition}.    By similar methods as in \cite{dipperdotystoll2} one can show that $\sum_{(\aa',\bb')\in M}v_{\aa',\bb'}m_{(r-k),(s-k)}$ involves only dominating elements (such that the partitions involve less boxes). In the general case let $i$ be again minimal violating the condition and let $l_1$ and $l_2$ be the number of entries $\leq i$ in the first row of $\aa$ and $\bb$ respectively. Then $m_{\lambda,\mu}$ has a left factor $m_{(l_1),(l_2)}$. Now, the result  follows by plugging  in the results for the special case and induction on the sum of all entries in the first rows. 
\end{proof}
The proof also shows the following proposition:
\begin{cor}
  Let $(\lambda,\mu)\in\Lambda_0(r,s)$. Let $V(\trianglerighteq(\lambda,\mu))$ be the $R$-span of the set 
 \[\{
   v_{(\aa,\bb),(\tt,\uu,\vv)}\mid 
   (\aa,\bb)\in \Rat(\lambda',\mu'),
   (\tt,\uu,\vv)\in M_0(\lambda',\mu'),(\lambda',\mu')\trianglerighteq(\lambda,\mu)\}\]
    and similarly, let $V(\rhd(\lambda,\mu))$ be the $R$-span of the set 
    \[\{
   v_{(\aa,\bb),(\tt,\uu,\vv)}\mid 
   (\aa,\bb)\in \Rat(\lambda',\mu'),
   (\tt,\uu,\vv)\in M_0(\lambda',\mu'),(\lambda',\mu')\rhd(\lambda,\mu)\}.\]
	Then $V(\trianglerighteq(\lambda,\mu))$ and $V(\rhd(\lambda,\mu))$ are 
	$U$-$B_{r,s}(n)$-submodules. 
	\end{cor}

It is also easy to see that for a fixed triple $(\tt,\uu,\vv)\in M_0(\lambda,\mu)$, 
\[\left\langle v_{(\aa,\bb),(\tt,\uu,\vv)}\mid 
   (\aa,\bb)\in \Rat(\lambda,\mu)\right\rangle_R+V(\rhd(\lambda,\mu))\] 
   is a $U$-submodule and thus the  following definition makes sense and is independent of the chosen triple $(\tt,\uu,\vv)$. 
\begin{defn}
For $(\lambda,\mu)\in\Lambda_0(r,s)$ choose a triple	 $(\tt,\uu,\vv)\in M_0(\lambda,\mu)$.
 Let $X^{(\lambda,\mu)}$ be the $U$-module 
 \[\left\langle v_{(\aa,\bb),(\tt,\uu,\vv)}\mid 
   (\aa,\bb)\in \Rat(\lambda,\mu)\right\rangle_R+V(\rhd(\lambda,\mu))/V(\rhd(\lambda,\mu)).\]
\end{defn}

We will show that $X^{(\lambda,\mu)}$ are dual Weyl modules.
Let $D^+$ and $D^-$ respectively  be the one-dimensional $U$-modules with basis $\{d^	+\}$ and $\{d^-\}$ on which the generators $E_i=\theta^i_{i+1}$ and $F_i=\theta^{i+1}_i$ act as zero, $\theta^i_i d^+=d^+$ and $\theta^i_id^-=-d^-$. Let $S:U\to U$ be the antipode of $U$, i.~e.~the antiisomorphism fixing $E_i$ and $F_i$ and mapping $\theta^i_i$ to $-\theta^i_i$. If $M$ is a $U$-module, then $M^*$ becomes a $U$-module setting $uf(m)=           f(S(u)m)$.

 Now,  $\iota:V^*\to V^{\otimes n-1}\otimes D^-:v_i\mapsto (-1)^i v_{(12\ldots\hat{i}\ldots n)}y_{(n-1)}\otimes d^-$ defines a $U$-monomorphism ($\hat{i}$ means omitting $i$). Likewise, one can define a $U$-monomorphism $\iota: V^{\otimes r}\otimes {V^*}^{\otimes s}\to V^{\otimes r+(n-1)s}\otimes {D^-}^{\otimes s}$. If $\{1,\ldots,n\}=\{i_1,\ldots,i_l\}\dot\cup\{j_1,\ldots,j_{n-l}\}$ is a disjoint union, then it can be seen easily, that
 \[\iota((v_{i_1}^*\otimes \cdots \otimes v_{i_l}^*)y_{(l)})=\pm (v_{j_{n-l}}\otimes \cdots\otimes v_{j_1})y_{(n-l)}\otimes ({v_{(1\ldots n)}y_{(n)}})^{\otimes l-1}\otimes {d^-}^{\otimes l}.\] Furthermore, $\iota(\sum_{i=1}^n v_i\otimes v_i^*)={v_{(1\ldots n)}y_{(n)}}$.
 
  If $(\aa,\bb)\in \Rat(\lambda,\mu)$, let $\tau$ be the partition from Remark~\ref{remark:bijection_rational} and $\cc$ be the standard $\tau$-tableau in bijection with $(\aa,\bb)$. Then for each  $(\tt,\uu,\vv)\in M_0(\lambda,\mu)$ there exists a permutation $\pi_{(\tt,\uu,\vv)}\in\mathfrak{S}_{r+(n-1)s}$ such that $\iota(v_{(\aa,\bb),(\tt,\uu,\vv)})=\pm v_{\cc,\tt^\tau}\pi_{(\tt,\uu,\vv)}\otimes {d^-}^{\otimes s}$.
This shows that $\iota(V(\rhd(\lambda,\mu)))\subseteq V(\rhd\tau)\otimes {D^-}^{\otimes s}$ and $X^{(\lambda,\mu)}\cong X^\tau \otimes {D^-}^{\otimes s}$ with $X^\tau=X^{(\tau,\emptyset)}$.

Let $\tau^t$ be the transpose of $\tau$. By inspecting the action of $U$ on $X^\tau$, one can see that $X^\tau$ is isomorphic to the module $D_{\tau^t,R}$ defined in \cite{green}. Note that the modules in \cite{green} can be viewed as modules over $U$, since the Schur algebras are epimorphic images of $U$. Furthermore, $ D_{\tau^t,R} \cong V_{\tau^t,R}^\circ$ (notation as in \cite{green}) is isomorphic to $  {}^\omega\Lambda_{\tau^t}^*\cong\Lambda_{-w_0 \tau^t}^*$ with $\omega$, $\Lambda_{\square}$ and $w_0$ defined as in \cite{lusztig}.

\begin{prop}\label{prop:Weyl_module}
Let $\lambda^t=(\lambda_1',\lambda_2',\ldots,\lambda_l')$ and $\mu^t=(\mu_1,\ldots,\mu_{l'}')$. Then 	we have
\[X^{(\lambda,\mu)}\cong \Lambda_{(\mu_1',\mu_2',\ldots,\mu_{l'}',0,\ldots,0,-\lambda_l',\ldots,-\lambda_2',-\lambda_1')}^*.\]
\end{prop} 
\begin{proof}
$X^{(\lambda,\mu)}\cong \Lambda_{-w_0 \tau^t}^*\otimes {D^-}^{\otimes s}\cong \Lambda_{-w_0 \tau^t}^*\otimes {{D^+}^*}^{\otimes s}\cong 
(\Lambda_{-w_0 \tau^t}\otimes {{D^+}}^{\otimes s})^*\cong 
\Lambda_{-w_0 \tau^t+s\cdot(1^n)}^*$ and the claim follows by computing $-w_0 \tau^t+s\cdot(1^n)$. 	
\end{proof}

The modules $\Lambda_{\square}$ are often (but not always) called Weyl modules, thus Propsition~\ref{prop:Weyl_module} shows that the modules $X^{(\lambda,\mu)}$ are dual Weyl modules where dual modules are defined via the antipode.

\begin{thm}[Filtration with cell modules and dual Weyl modules]\label{thm:filtration}
 We have
	\[
	V(\trianglerighteq(\lambda,\mu))/V(\rhd(\lambda,\mu))\cong X^{(\lambda,\mu)}\otimes\overline{C}^{(\lambda,\mu)}
	\]
	as $U$-$B_{r,s}(n)$-bimodules. In particular, $V^{\otimes r}\otimes {V^*}^{\otimes
s}$ has a filtration with cell modules and one with dual Weyl modules. 
   
\end{thm}
\begin{proof} 
This is straight forward from the construction of the basis. 
\end{proof}